\documentclass[11pt]{amsart}
\usepackage[utf8]{inputenc}
\usepackage{amssymb,amsmath}
\usepackage{graphicx}
\usepackage{color}
\usepackage{tikz}
\usepackage{pgfplots}
\usepackage{enumitem}
\usepackage{algorithm}
\usepackage{setspace}
\usepackage{algpseudocode}
\usepackage{mathtools}
\usepackage{xfrac}

\usepackage{accents}
\usepackage{multicol} 
\usepackage{soul}
\usepackage{subcaption}
\usepackage{booktabs}
\usepackage{array}
\usepackage{comment}
\newcolumntype{L}[1]{>{\raggedright\let\newline\\\arraybackslash\hspace{0pt}}m{#1}}
\newcolumntype{C}[1]{>{\centering\let\newline\\\arraybackslash\hspace{0pt}}m{#1}}
\newcolumntype{R}[1]{>{\raggedleft\let\newline\\\arraybackslash\hspace{0pt}}m{#1}}
\usepackage{siunitx}
\usepackage{bbold}

\usepackage[hidelinks]{hyperref}
\usepackage[nameinlink,noabbrev]{cleveref}
\newtheorem{theorem}{Theorem}
\newtheorem{proposition}[theorem]{Proposition}

\theoremstyle{definition}

\newtheorem{example}[theorem]{Example}
\theoremstyle{lemma}
\newtheorem{lemma}[theorem]{Lemma}

\theoremstyle{remark}
\newtheorem{remark}[theorem]{Remark}

\newtheorem{assumption}[theorem]{Assumption}
\Crefname{assumption}{Assumption}{Assumptions}
\numberwithin{theorem}{section}
\numberwithin{equation}{section}
\numberwithin{table}{section}
\numberwithin{figure}{section}
\definecolor{myBlue}{RGB}{30,144,255} % dodger blue
\definecolor{myGreen}{RGB}{69,169,0} % chatreuse
\definecolor{myRed}{RGB}{165,12,42} 
\definecolor{myOrange}{RGB}{225,92,22} 
\definecolor{color0}{rgb}{0.12156862745098,0.466666666666667,0.705882352941177}
\definecolor{color1}{rgb}{1,0.498039215686275,0.0549019607843137}
\definecolor{color2}{rgb}{0.172549019607843,0.627450980392157,0.172549019607843}
\definecolor{color3}{rgb}{0.83921568627451,0.152941176470588,0.156862745098039}
\definecolor{color4}{rgb}{0.580392156862745,0.403921568627451,0.741176470588235}
\definecolor{color5}{rgb}{0,0,0}

\newcommand{\delete}[1]{ }

\def\N{\mathbb{N}}

\def\R{\mathbb{R}}

\newcommand\ds{\,\mathrm{d}s}
\newcommand\dt{\,\mathrm{d}t}
\newcommand\dx{\,\mathrm{d}x}

\newcommand{\tddt}{\ensuremath{\tfrac{\mathrm{d}}{\mathrm{d}t}} }

\newcommand{\tdds}{\ensuremath{\tfrac{\mathrm{d}}{\mathrm{d}s}} }

\newcommand{\hook}{\ensuremath{\hookrightarrow}}

\DeclareMathOperator{\id}{id}

\newcommand{\A}{\ensuremath{\mathcal{A}} }
\newcommand{\calA}{\ensuremath{\mathcal{A}} }

\newcommand{\calB}{\ensuremath{\mathcal{B}} }
\newcommand{\calC}{\ensuremath{\mathcal{C}} }

\newcommand{\calD}{\ensuremath{\mathcal{D}} }

\newcommand{\cHV}{\ensuremath{{\mathcal{H}_{\V}}} }
\newcommand{\cHQ}{\ensuremath{{\mathcal{H}_{\Q}}} }

\newcommand{\Q}{\ensuremath{\mathcal{Q}} }
\newcommand{\calQ}{\ensuremath{\mathcal{Q}} }

\newcommand{\V}{\ensuremath{\mathcal{V}}}
\newcommand{\calV}{\ensuremath{\mathcal{V}}}
\newcommand{\calX}{\ensuremath{\mathcal{X}} }

\newcommand{\f}{\ensuremath{f}}
\newcommand{\g}{\ensuremath{g}}

\newcommand{\proj}{R}
\newcommand{\projInd}[1]{\ensuremath{\proj}_{#1}}

\newcommand{\Ru}{\projInd{a}}
\newcommand{\Rp}{\projInd{b}}

\newcommand{\CQtoH}{\ensuremath{{C_{\scalebox{.5}{\Q \hook \cHQ}}}} }
\newcommand{\CQtoHsquare}{\ensuremath{{C^2_{\scalebox{.5}{\Q \hook \cHQ}}}} }

\newcommand{\D}{D_\tau}

\DeclareFontFamily{U}{matha}{\hyphenchar\font45}
\DeclareFontShape{U}{matha}{m}{n}{
	<-6> matha5 <6-7> matha6 <7-8> matha7
	<8-9> matha8 <9-10> matha9
	<10-12> matha10 <12-> matha12
}{}
\DeclareSymbolFont{matha}{U}{matha}{m}{n}

\DeclareFontFamily{U}{mathx}{\hyphenchar\font45}
\DeclareFontShape{U}{mathx}{m}{n}{
	<-6> mathx5 <6-7> mathx6 <7-8> mathx7
	<8-9> mathx8 <9-10> mathx9
	<10-12> mathx10 <12-> mathx12
}{}
\DeclareSymbolFont{mathx}{U}{mathx}{m}{n}

\DeclareMathDelimiter{\vvvert} {0}{matha}{"7E}{mathx}{"17}%
\allowdisplaybreaks
\begin{document}
\title[]{A decoupling and linearizing discretization\\ for weakly coupled poroelasticity\\ with nonlinear permeability}
\author[]{R.~Altmann$^\dagger$, R.~Maier$^\ddagger$}
\address{${}^{\dagger}$ Department of Mathematics, University of Augsburg, Universit\"atsstr.~14, 86159 Augsburg, Germany}
\email{robert.altmann@math.uni-augsburg.de}
\address{${}^{\ddagger}$ Department of Mathematical Sciences, Chalmers University of Technology and University of Gothenburg, 412 96 G\"oteborg, Sweden}
\email{roland.maier@chalmers.se}
%\thanks{}
%
\date{\today}
%
%
%=============================================================================
%=========  Abstract
%=============================================================================
\begin{abstract}
We analyze a semi-explicit time discretization scheme of first order for poro\-elasticity with nonlinear permeability provided that the elasticity model and the flow equation are only weakly coupled. The approach leads to a decoupling of the equations and, at the same time, linearizes the nonlinearity without the need of further inner iteration steps. Hence, the computational speed-up is twofold without a loss in the convergence rate. We prove optimal first-order error estimates by considering a related delay system and investigate the method numerically for different examples with various types of nonlinear displacement-permeability relations. 
\end{abstract}
%
%
%=============================================================================
%=========  Title / Contents
%=============================================================================
\maketitle
%\setcounter{tocdepth}{3}
%\tableofcontents
%
%
%65M12 = Stability and convergence of numerical methods PDEs
%65J15 = Numerical solutions to equations with nonlinear operators 
%76S05 = Flows in porous media; filtration
% 
%
%=============================================================================
%=========  Introduction
%=============================================================================
\section{Introduction}
% Poroelasticity and applications
In this paper, we consider Biot's quasi-static equations of poroelasticity, which couple a Darcy flow of an incompressible viscous fluid to the linear elastic behavior of the surrounding porous media~\cite{Biot41}. The setup is of particular importance for physical applications, e.g., in the field of geo\-mechanics~\cite{Zob10}, and more recently also found success in connection with medical applications~\cite{TulV11,SobEWC12,VarCTHLTV16}. In this context, the extension of the classical Biot's equations to a multiple-network case (see, e.g., \cite{LeePMR19,HonKLP19}) has also gained increasing interest. 

The poroelastic model relies on an averaging of the pressure and the displacement field on very small volumetric elements and assumes small strains and a quasi-static behavior, i.e., internal equilibria are preserved at any time. 
% nonlinear permeability
In the context of a constant permeability, the model is analyzed in~\cite{DetC93,Sho00}. However, depending on the particular model configuration, the permeability may depend on the porosity, as, e.g, investigated by Kozeny~\cite{Koz27} and Carman~\cite{Car37,Car38}. In the context of poroelasticity, the porosity, in turn, depends (in a nonlinear way) on the displacement, which leads to a nonlinear poroelastic model as analyzed in~\cite{CaoCM13} and, in a more general setting, in~\cite{BocGSW16}. 

% examples for discretization
Due to the present nonlinearity, linearization approaches are required to solve the nonlinear system that occurs in every time step of an implicit time discretization. A prominent example is the application of a Picard-type iteration as used, for instance, in~\cite{CaoCM13,BroV16b,FuCM20}. For strong nonlinearities, where multiple inner steps are required, this leads to relatively expensive computations. 
One important aspect in this context is the coupling of the involved equations. In the context of linear problems, this can be overcome with a suitable decoupling. One approach is the so-called \emph{fixed-stress splitting scheme} considered in~\cite{WheG07,KimTJ11,MikW13}. From a theoretical point of view, the resulting numerical methods require an additional iteration in each time step, see, e.g., \cite{StoBKNR19}. In practice, however, one can restrict the inner iteration to very few steps; cf.~\cite{KimTJ11}. The ideas of a 'fixed stress' have also been used in the context of a nonlinear poroelastic model in~\cite{BorRKN18}. An approach without an inner iteration is presented in~\cite{ChaR18} and is similar to the here considered approach of a time delay, but uses additional stabilization terms in the decoupled equations. 

% two-fold advantage of weak coupling
In this contribution, we investigate the use of a semi-explicit time discretization scheme in the setting where the coupling between the two poroelastic equations is rather small; cf.~the rigorous condition given in Assumption~\ref{ass:weakCoupling}. We refer to this particular configuration as \emph{weak coupling}, which is satisfied in many applications. Note that a semi-explicit discretization is related to the so-called \emph{drained splitting scheme} (see, e.g., \cite{ArmS92,KimTJ11b}) if only one inner iteration step is computed. The advantage of a semi-explicit discretization for the classical (linear) poroelastic model was studied in~\cite{AltMU21}. Therein it was shown that a semi-explicit Euler discretization maintains the first-order convergence and at the same time decouples the system equations. This decoupling of the flow and the elasticity equations can enable a remarkable computational speed-up compared to a classical implicit discretization. 
In the here considered nonlinear setting, this speed-up is further enhanced by the fact that the semi-explicit discretization naturally leads to a linearization of the nonlinear term, which does not require additional inner iterations to handle the nonlinearity. 

Since we are mainly interested in investigating the advantages of using a semi-explicit discretization, we put the theoretical focus on the temporal discretization. This means, in particular, that we present the convergence analysis of the semi-explicit scheme without spatial approximations. Note, however, that the given convergence proof can be extended to the fully discrete setting. For this, two separate Ritz projections with appropriate approximation properties need to be introduced. As we will indicate, the convergence proof then follows in a similar manner. 

% outline 
The remaining parts of the paper are structured as follows. In Section~\ref{sec:poroNL}, we state the nonlinear poroelasticity model and introduce necessary assumptions. Moreover, several examples of displacement-dependent nonlinear permeabilities are discussed. 
We then investigate the semi-explicit discretization scheme, leading to (semi-discrete) error estimates in~\Cref{sec:discretization}. For this, we introduce a related delay model, where the time delay equals the step size of the temporal discretization. We analyze the semi-explicit scheme based on the theoretical observation that it can be reinterpreted as an implicit discretization of the delay model. First-order convergence can then be shown under the assumption of the weak coupling condition, a sufficiently small step size, and a certain regularity of the solutions to the original and the delay system. 
Finally, we assess the scheme numerically in Section~\ref{sec:num}, where we consider several nonlinear permeability models. \\

%\subsection*{Notation} 
\textbf{Notation.} 
Throughout the paper we write~$a \lesssim b$ to indicate that there exists a generic constant~$C$, independent of spatial or temporal discretization parameters, such that~$a \leq C b$. Further, we abbreviate Bochner spaces (especially when considering the norms) on the time interval~$[0,T]$ for a Banach space~$\calX$ by~$L^p(\calX) \vcentcolon = L^p(0,T;\calX)$, $W^{k,p}(\calX) \vcentcolon = W^{k,p}(0,T;\calX)$, and~$H^k(\calX) \vcentcolon = H^k(0,T;\calX)$, $p\ge 1$, $k\in\N$. 
%
%
%=============================================================================
%=========  Nonlinear poroelasticity
%=============================================================================
\section{Nonlinear poroelasticity}\label{sec:poroNL}
Let $D \subseteq \R^d$, $d \in \{2,3\}$, be a domain with Lipschitz boundary $\partial D$. 
Generally, the poroelastic equations seek the pore pressure ${p\colon[0,T] \times D \rightarrow \mathbb{R}}$ and the displacement field $u\colon [0,T] \times D \rightarrow \mathbb{R}^d$ up to a given final time $T>0$ such~that 
\begin{subequations}
\label{eq:poroStrong}
\begin{align}
	-\nabla \cdot \sigma(u) + \nabla (\alpha p) &\;=\; \f\qquad \text{in } (0,T] \times D, \label{eq:poroStrong:a} \\
	\partial_t\Big(\alpha \nabla \cdot u + \frac{1}{M} p \Big)- \nabla \cdot \Big( \frac{\kappa(\nabla\cdot u)}{\nu} \nabla p \Big) &\;=\; \g\qquad \text{in } (0,T] \times D. \label{eq:poroStrong:b}
\end{align}
Here, $\f$ is a volumetric load, $\g$ the fluid source, and $\sigma$ denotes the \emph{stress tensor}. 
The permeability~$\kappa$ is the source of a possible nonlinearity that will be discussed in the following subsections. 
The remaining constants read $\alpha$ (\emph{Biot-Willis fluid-solid coupling coefficient}), $M$ (\emph{Biot modulus}), and $\nu$ (\emph{fluid viscosity}); see~\cite{Biot41,Sho00} for further details. 	
For simplicity, we consider homogeneous Dirichlet boundary conditions, i.e., 
\begin{align}
	u &\;=\; 0 \quad\ \text{on } (0,T] \times \partial D, \label{eqn:bcs:u} \\
	p &\;=\; 0 \quad\ \text{on } (0,T] \times \partial D. \label{eqn:bcs:p}
\end{align}
\end{subequations}
More general inhomogeneous boundary conditions may be incorporated in form of a constraint~\cite{AltMU21b}. 
As initial condition, we have $p(\,\cdot\,,0) = p^0$, which also defines $u(\,\cdot\,,0)$ due to~\eqref{eq:poroStrong:a}. 
%
%
%=============================================================================
\subsection{Linear model and weak form}\label{sec:poroNL:linear}
If the stress-strain constitutive law given by~$\sigma$ is linear and the permeability~$\kappa$ is constant, then system~\eqref{eq:poroStrong} is called \emph{linear poroelasticity}. In this case, $\sigma$ has the form 
\begin{equation*}
%\label{eq:linStrainStress}
  \sigma(u) 
  \coloneqq 2\mu \,\varepsilon (u) + \lambda\, (\nabla \cdot u) \,\id,
\end{equation*}
where $\mu$ and $\lambda$ are the \emph{Lam\'e coefficients}, $\id$ is the identity tensor, and
\begin{equation*}
\label{eq:symStrainGrad}
\varepsilon (u) = \frac{1}{2} \left(\nabla u + (\nabla u)^T\right)
\end{equation*}
is the \emph{symmetric strain gradient}. Note that the Lam\'e coefficients may explicitly depend on the spatial variables, causing multiscale effects~\cite{BroV16b,FuACMPP19,AltCMPP20}. 

% weak formulation
In view of the numerical approximation of the solution, we consider the corresponding weak formulation. For this, we introduce the Hilbert spaces 
\[
\V:=[H^1_0(\Omega)]^d, \qquad
\Q:=H^1_0(\Omega)
\]
as ansatz spaces for $u$ and $p$, respectively. Accordingly, we define the two pivot spaces~$\cHV:=[L^2(\Omega)]^d$ and~$\cHQ:=L^2(\Omega)$ such that~$\V,\cHV,\V^*$ and~$\Q,\cHQ,\Q^*$ define Gelfand triples with dense embeddings, see~\cite[Ch.~23.4]{Zei90a}. 
The $L^2$-inner products corresponding to~$\cHV$ and $\cHQ$ are simply denoted by~$(\,\cdot\,,\cdot\,)$. 

% bilinear forms 
As a second step, we introduce bilinear forms corresponding to the (differential) operators in~\eqref{eq:poroStrong}. More precisely, we define~$a\colon\V\times\V\to\R$, $b, c\colon\Q\times\Q\to\R$, and $d\colon\V\times\Q\to\R$ by 
\begin{align*}
  a(u,v) \coloneqq \int_\Omega \sigma(u) : \varepsilon(v) \dx, \qquad 
  b(p,q) \coloneqq \int_\Omega \frac{\kappa}{\nu}\, \nabla p \cdot \nabla q \dx,\\
  c(p,q) \coloneqq \int_\Omega \frac{1}{M}\, p\, q \dx, \qquad 
  d(u,q) \coloneqq \int_\Omega \alpha\, (\nabla \cdot u)\, q \dx 
\end{align*}
for $u,\, v \in \V$ and $p,\, q \in \Q$.
Then, the weak formulation of system~\eqref{eq:poroStrong} seeks abstract functions~$u\colon [0,T] \to \V$ and $p\colon [0,T] \to \Q$ such that
\begin{subequations}
\label{eq:poroBilinear}
\begin{align}
	a(u,v) - d(v, p) 
	&= (\f, v), \label{eq:poroBilinear:a} \\
	d(\dot u, q) + c(\dot p,q) + b(p,q) 
	&= (\g, q).  
	\label{eq:poroBilinear:b} 
\end{align}
\end{subequations}	
for all test functions $v \in \V,\, q \in \Q$. 
Results on the unique solvability of the system in terms of a strong solution are shown in~\cite{Sho00}. % Zen84 auch nicht wirklich weak 
We will discuss the existence of weak solutions in~\Cref{prop:weakSlnLinearCase} later on. 
\begin{remark}
System~\eqref{eq:poroBilinear} can be interpreted as a coupled system of an elliptic and a parabolic equation such that a spatial discretization yields a coupled system of an algebraic and a differential equation. As a consequence, system~\eqref{eq:poroBilinear} equals a so-called \emph{partial differential-algebraic equation}. 
Furthermore, one can show that the poroelastic equations~\eqref{eq:poroBilinear} have a port-Hamiltonian structure; cf.~\cite{AltMU21b}. 
\end{remark}
%
%
%=============================================================================
\subsection{General setting}\label{sec:poroNL:assumptions}
Before we introduce the displacement-dependent nonlinear permeability, the aim of this subsection is to gather properties of the introduced bilinear forms for the linear case. Further, we discuss the existence of weak solutions and the \emph{weak coupling} assumption. 

% bilinear forms a,b,c 
The bilinear form~$a\colon\V\times\V\to\R$ contains the linear elasticity model. Thus, by Korn's inequality we know that~$a$ is elliptic with a constant~$c_a$ that is mainly characterized by~$\mu$; cf.~\cite[Sect.~6.3]{Cia88}. Further, $a$ is symmetric and bounded in~$\V$. This implies that~$\Vert\cdot\Vert_a \vcentcolon= a(\,\cdot\,, \cdot\,)^{1/2}$ defines a norm, which is equivalent to the~$\V$-norm, namely 
\[
  c_a \|u\|^2_\V
  \le \|u\|^2_a
  = a(u,u)
  \le C^2_a\, \|u\|^2_\V. 
\]
Here, $C_a$ denotes the continuity constant of~$a$.  
Similarly, with a constant and positive permeability~$\kappa$ one has in the linear model that~$b\colon \Q\times\Q\to \R$ is symmetric, elliptic, and bounded in~$\Q$. 
The bilinear form~$c\colon \Q\times\Q\to \R$ equals a scaled variant of the $L^2$-inner product. Hence, it is symmetric, elliptic, and bounded in the pivot space~$\cHQ$. As a result, $\Vert\cdot\Vert_c \vcentcolon= c(\,\cdot\,, \cdot\,)^{1/2}$ defines a norm, which is equivalent to the~$\cHQ$-norm. 

% coupling term d 
The bilinear form~$d\colon \V\times\Q\to \R$ describes the coupling of the elliptic and parabolic part of the poroelastic system. Since we may integrate by parts, there are two continuity estimates, namely 
\begin{equation*}
  d(u,p) 
  \le C_d\, \Vert u \Vert_\V \Vert p\Vert_\cHQ
  \qquad\text{and}\qquad
  d(u,p) 
  \le \tilde C_d\, \Vert u \Vert_\cHV \Vert p\Vert_\Q.
\end{equation*}
With these assumptions we can prove the following existence result. 
\begin{proposition}[Weak solution for the linear case]
\label{prop:weakSlnLinearCase}
Consider initial data~$p^0\in \cHQ$ and right-hand sides $\f\in H^1(0,T;\V^*)$, $\g\in L^2(0,T;\Q^*)$. 
Under the above assumptions on the bilinear forms~$a$, $b$, $c$, and~$d$, system~\eqref{eq:poroBilinear} has a unique weak solution
\begin{align*}
  p \in L^2(0,T;\Q) \cap H^1(0,T;\Q^*) \hookrightarrow C([0,T];\cHQ), \qquad
  u \in L^2(0,T;\V).
\end{align*}
\end{proposition}
\begin{proof} 	
We only give the main ideas of the proof. For this, let~$\calA\colon\V\to\V^*$, $\calB,\calC\colon\Q\to\Q^*$, and~$\calD\colon\V\to\Q^*$ denote the operators corresponding to the bilinear forms $a$, $b$, $c$, and~$d$, respectively. Since~$\calA$ is invertible, we can differentiate the first equation in time and insert it into the second. This then yields the (linear) parabolic equation~$( \calC + \calD \calA^{-1}\calD^* )\, \dot p + \calB p = \g - \calD\calA^{-1} \dot\f$. Hence, the existence of a unique solution~$p$ in the stated spaces follows by~\cite[Ch.~3, Sect.~4]{LioM72}. The first equation finally yields the unique solution~$u$. 	
%%	
%% extended proof: 	
%To shorten the proof, let~$\calA\colon\V\to\V^*$, $\calB,\calC\colon\Q\to\Q^*$, and~$\calD\colon\V\to\Q^*$ denote the operators corresponding to the bilinear forms $a$, $b$, $c$, and~$d$, respectively. 
%With this, system~\eqref{eq:poroBilinear} can equivalently be written in the form
%%
%\[
%\calA u - \calD^* p = \f\ \text{ in }\V^*, \qquad
%\calD \dot{u} + \calC \dot{p} + \calB p = \g \ \text{ in }\Q^*.
%\]
%%
%Since~$\calA$ is invertible, we can (formally) differentiate the first equation in time and insert it into the second. This then yields the equation 
%%
%\[
%  ( \calC + \calD \calA^{-1}\calD^* )\, \dot p + \calB p = \g - \calD\calA^{-1} \dot\f. 
%\]
%%
%It remains to show that this is a (linear) parabolic equation, which then implies the existence of a unique solution~$p \in L^2(0,T;\Q)$ with~$\dot p \in L^2(0,T;\Q^*)$; cf.~\cite[Ch.~3, Sect.~4]{LioM72}. Moreover, by the first equation, we conclude the existence of~$u$. The second equation, in turn, shows that~$\calD \dot u \in L^2(0,T;\Q^*)$. 
%
%By assumption, we know that~$\calB$ is elliptic. We show that~$\tilde{\calC} := \calC + \calD \calA^{-1}\calD^*$ defines an inner product in~$\cHQ$. Since~$\calC$ is assumed to be elliptic, it is sufficient to note that for~$p\in\cHQ$ we have~$\langle \calD \calA^{-1}\calD^* p, p\rangle = \langle \calA^{-1}\calD^* p, \calD^* p\rangle \ge 0$.
\end{proof}
For the convergence analysis of Section~\ref{sec:discretization} we need the assumption that the coupling is sufficiently weak in the following sense. 
\begin{assumption}[weak coupling condition]
\label{ass:weakCoupling}
We assume a weak coupling of the form 
\begin{align*}
	C_d^2 \le c_a\, c_c,  
\end{align*}
where $c_a, c_c$ denote the ellipticity constants of the bilinear forms~$a$, $c$, respectively, and~$C_d$ the continuity constant of~$d$. 	
\end{assumption}
\begin{remark}
If the physical coefficients introduced in Section~\ref{sec:poroNL:linear} are assumed to be constant and $\|\cdot\|_\V := \|\nabla \cdot\|_{L^2(D)}$, then the coupling condition in Assumption~\ref{ass:weakCoupling} may be rephrased as
\begin{equation*}
\alpha^2 M/\mu \leq 1,
\end{equation*}
see, e.g., \cite[Sect.~6.3]{Cia88}. We emphasize that this is essentially the coupling condition that is required for the drained splitting approach to be stable; cf.~\cite{KimTJ11b}. 
	
Note that in the field of poroelasticity, the coefficient~$\alpha$ is generally bounded by $1$. Further, the quotient $M/\mu$ is often of order $1$ as well; cf.~\cite[Sect.~3.3.4]{DetC93}. Therefore, Assumption~\ref{ass:weakCoupling} is satisfied in certain applications as they appear, e.g., in the field of geomechanics. 
\end{remark}
In the following, we discuss a nonlinear extension of the poroelastic equations, where we allow the bilinear form $b$ to depend on the displacement. 
%
%
%=============================================================================
\subsection{Nonlinear displacement-dependent permeability}\label{sec:poroNL:nonlinear}
We now leave the linear setting and assume that the permeability~$\kappa$ depends (in a possibly nonlinear fashion) on the divergence of the displacement. Thus, we replace the previously defined bilinear form $b$ by  
\begin{equation}
\label{def:nonlinearB}
  b(u;p,q) 
  \coloneqq \int_\Omega \frac{\kappa(\nabla\cdot u)}{\nu}\, \nabla p \cdot \nabla q \dx.
\end{equation}
Note that for fixed~$u$, $b$ has the same structure as in the linear case. 
Here, however, we consider the nonlinear case, which leads to the system 
\begin{subequations}
\label{eq:poroNonlinear}
\begin{align}
	a(u,v) - d(v, p) 
	&= (\f, v), \label{eq:poroNonlinear:a} \\
	d(\dot u, q) + c(\dot p,q) + b(u;p,q) 
	&= (\g, q)  
	\label{eq:poroNonlinear:b} 
\end{align}
\end{subequations}	
for test functions $v\in\V$, $q\in\Q$ and with the initial condition~$p(0)=p^0$. 
For the nonlinearity, we make the following assumptions. 
\begin{assumption}[nonlinear permeability]
\label{ass:b}
We assume that~$\kappa$ is Lipschitz continuous and bounded in terms of constants $\kappa_-$ and $\kappa_+$, i.e.,
\begin{align}
\label{eq:boundKappa}
  0 < \kappa_- \le \kappa \le \kappa_+ < \infty. 
\end{align}
This implies the existence of a positive constant~$L_b$ (depending on~$\nu$, $\kappa_+$, and the Lipschitz constant of $\kappa$) such that~$b\colon \V\times\Q\times\Q\to \R$ introduced in~\eqref{def:nonlinearB} satisfies 
\begin{subequations}\label{eq:boundB}
\begin{align}
  \| b(u;p,\cdot) - b(u;q,\cdot) \|_{\Q^*} 
  &\le L_b\, \| p-q \|_\Q, \label{eq:boundB1}\\
  \| b(u;p,\cdot) - b(v;p,\cdot) \|_{\Q^*} 
  &\le L_b\, \|p\|_\Q\,\| u-v \|_\V \label{eq:boundB2}
\end{align}
\end{subequations}
for all $u,v\in\V$ and $p,q\in\Q$. 
Furthermore, the lower bound of~$\kappa$ yields 
\begin{align}\label{eq:coerciveB}
  b(u; p, p)
  \geq c_b\, \|p\|_\Q^2 
\end{align}
uniformly in~$u\in\V$ and for all~$p\in\Q$, where $c_b$ depends on~$\nu$ and~$\kappa_-$. 
\end{assumption}
We emphasize that the remaining bilinear forms $a$, $c$, and $d$ remain unchanged, i.e., we keep the assumptions from Section~\ref{sec:poroNL:linear}.  
Further note that Assumption~\ref{ass:b} is a generalization of the linear case presented in Section~\ref{sec:poroNL:linear}, since the assumptions simplify to the continuity and ellipticity of the operator in the constant case~$\kappa(\nabla\cdot u)\equiv \kappa$. 

In the outlined setting, we can conclude the solvability of the nonlinear poroelasticity problem. Uniqueness is discussed afterwards. 
\begin{proposition}[Solvability of the nonlinear system]	
Consider the setting of Section~\ref{sec:poroNL:linear} with $b$ defined in~\eqref{def:nonlinearB} satisfying Assumption~\ref{ass:b} and~$f\equiv0$. Further, let the initial data satisfy~$p^0\in\cHQ$ and~$\g\in L^2(0,T;\cHQ)$. 
% Remark: g mit Werten in L2 kommt aus CCM13-Paper
Then, system~\eqref{eq:poroNonlinear} has a weak solution 
\[
  u\in L^2(0,T; \calV), \qquad
  p\in L^2(0,T; \calQ).	
\]
\end{proposition}
\begin{proof}
The existence of a solution has been shown in~\cite[Th.~2.9]{CaoCM13} for the slightly more general case of~$\kappa$ being continuous rather than Lipschitz continuous.  
\end{proof}
\begin{remark}[Uniqueness of solutions]
\label{rem:uniqueness}
%Uniqueness conditions from \cite[Th.~2.10]{CaoCM13} 
The uniqueness of a solution to~\eqref{eq:poroNonlinear} is investigated in~\cite[Th.~2.10]{CaoCM13} and requires additional assumptions such as the Lipschitz continuity of~$\kappa$. Further, it is asked for a condition, which is similar to our weak coupling condition stated in Assumption~\ref{ass:weakCoupling}. 
Translated to our notion, the sufficient condition reads 
\begin{equation*}
C_d^2\, 
\big( C_\mathrm{P} C_{\calD\A^{-1}\calD^*} \tfrac{\kappa_+}{\kappa_-} \big)
< c_a\, c_c,
\end{equation*}
where~$C_\mathrm{P}$ equals the Poincar{\'e} constant within~$\|q\|_{H^1(\Omega)} \leq C_\mathrm{P}\,\|\nabla q\|_{L^2(\Omega)}$ for~$q \in \Q$ and $C_{\calD\A^{-1}\calD^*}$ denotes the $H^1$-continuity constant of the operator $\calD\A^{-1}\calD^*$ (using the notation from the proof of Proposition~\ref{prop:weakSlnLinearCase}). 
% Remark: Beweis in CCM braucht $H^2$-Regularitaet fuer $u$ in der ersten Gleichung 
Hence, we recover the weak coupling condition up to a constant factor, which depends on~$\kappa$. 
The final condition for the unique solvability of~\eqref{eq:poroNonlinear} is an~$L^\infty$-bound of~$\nabla p$, which can be achieved considering sufficiently smooth data. 
\end{remark}
In the remainder of this paper, we will always assume the existence of a unique solution~$(u,p)$ to~\eqref{eq:poroNonlinear} and concentrate on the numerical approximation of this solution. 
%
%
%=============================================================================
\subsection{Particular choices of the displacement-dependence in the permeability}\label{sec:nonlinModels}
To understand the explicit dependence of the permeability $\kappa$ on the divergence of the displacement, we emphasize that in many interesting configurations, the permeability may be assumed to depend on the \emph{porosity}, which describes the volume fraction which is occupied by the fluid. 
A well-established and in many cases reliable hypothesis for the explicit dependence on the porosity is the so-called \emph{Kozeny-Carman} relation, which traces back to Kozeny~\cite{Koz27} and was later on adjusted by Carman~\cite{Car37,Car38}. It couples the permeability in a nonlinear fashion to the porosity. For the particular case of a flow of a Newtonian fluid in the interstice between spherical particles in the context of poroelasticity, such a behavior is used and justified in~\cite{HsuC90}. The relation is further applied in~\cite{KimP99} to derive a permeability-displacement dependence under the assumption that the solid grains are relatively incompressible compared to the solid skeleton of the porous medium. To achieve this, a relation of the porosity and the \emph{volume strain} or \emph{dilatation} is required. The dilatation expresses the change of volume of the fluid and (under the assumption of small strains) is given by $\nabla \cdot u$. This explains the specific form of the permeability given in~\eqref{def:nonlinearB}, which is considered throughout this work.

A particular model for the permeability-displacement relation, which includes the Kozeny-Carman relation, is presented in the following example. 
\begin{example}[Kozeny-Carman-type permeability]
\label{exp:dilatationDependentPermeability}
A prominent case of a displacement-dependent permeability through the dilatation is the above-mentioned Kozeny-Carman-type permeability as considered in \cite{CaoCM13,BocGSW16} in the context of a poroelastic problem with linear stress-strain relations. % as in~\eqref{eq:linStrainStress}.
The particular choice of~$\kappa$ used therein reads 
\begin{equation*}
\kappa(s) \coloneqq \left\{
\begin{aligned}
&\,\kappa_-, \quad && s \leq c_s,\\
&\,\kappa_0 \,\frac{\rho^3(s)}{(1- \rho(s))^2}, \quad && c_s < s < C_s,\\
&\,\kappa_+, \quad && s \geq C_s,
\end{aligned}
\right.
\end{equation*}
where $\kappa_0$ is the initial saturated permeability and $\rho$ the porosity, i.e., the ratio between fluid volume and total volume given by
\begin{equation*}
\rho(s) = \rho_0 + (1-\rho_0)s,
\end{equation*}
where $\rho_0 \in (0,1)$ is a given ground porosity. 
Further, $c_s$ and $C_s$ are some prescribed lower and upper bounds which fulfill
\begin{equation*}
\frac{\rho_0}{\rho_0-1} < c_s < C_s < 1
\end{equation*}
and lead to 
\begin{equation*}
\kappa_- = \kappa_0 \,\frac{\rho^3(c_s)}{(1- \rho(c_s))^2},\qquad \kappa_+ = \kappa_0 \,\frac{\rho^3(C_s)}{(1- \rho(C_s))^2}.
\end{equation*}
We emphasize that this particular example satisfies Assumption~\ref{ass:b} due to the explicit bounds and the Lipschitz continuity of~$\kappa$.  
\end{example} 

Apart from the presented Kozeny-Carman relation, also other representations of the permeability exist, which might depend on the geometrical setup or the specific problem configuration; cf., e.g., the discussion in~\cite{SchRZRK19}. Another coupling of porosity and permeability is, for instance, considered in~\cite{CauGHPST14} to model blood flow in the lamina cribrosa. Therein, the permeability depends quadratically on the dilatation, which turns out to be suitable to describe the flow of a Newtonian fluid through cylindrical pores. 
Finally, we also mention the \emph{network-inspired} permeability as described in~\cite{RahLVM20} and the exponential dependence on the displacement as investigated in~\cite{LaiM80,HolM90}, again in the context of biological tissues.

The network-inspired model is described in the following example with an artificially introduced lower bound to fulfill Assumption~\ref{ass:b}.
\begin{example}[Network-inspired permeability]\label{exp:networkInspiredPermeability}
In network structures consisting of channels that can be open or closed, the flow rate depends on the number of channels that are open, see, e.g.~\cite{Bal87,Won88}. Considering the possibility of channels being randomly closed or open, a permeability-porosity relation is presented in~\cite{RahLVM20} for an arbitrary network topology that reads
\begin{equation*}
\hat\kappa(s) \coloneqq \left\{
\begin{aligned}
&\,0, \quad && \rho(s) < \hat\rho,\\
&\,\kappa_0\frac{\rho(s)-\hat\rho}{\rho_0-\hat\rho}, \quad && \rho(s) \geq \hat\rho,
\end{aligned}
\right.
\end{equation*}
where the porosity $\rho$ is given by
\begin{equation*}
\rho(s) = 1 - (1-\rho_0)\exp(-s)
\end{equation*}
with $\rho_0,\,\hat\rho \in (0,1)$ and $\hat\rho < \rho_0$.
In order to fulfill the lower bound in Assumption~\ref{ass:b}, we artificially introduce a threshold in the permeability and define
\begin{equation*}
\kappa(s) = \hat\kappa(s) + \kappa_0\,\delta
\end{equation*}
with some small $\delta > 0$. This ensures that a flow through the medium is always possible. 
\end{example}
\begin{remark}[Nonlinear thermoelasticity]
Besides the here considered poroelasticity models, the given framework also fits to problems in the field of thermoelasticity, which covers the displacement of a material due to temperature changes. 	
Since linear poroelasticity and linear thermoelasticity are equivalent from a mathematical point of view~\cite{Bio56}, the given setting also includes applications with a nonlinear heat conductivity.
Note that, in order to satisfy Assumption~\ref{ass:weakCoupling}, the thermal expansion coefficient needs to be much smaller than the stress tensor; cf.~\cite{CalR14}.	
\end{remark}
%
%
%
%=============================================================================
%=========  Discretization 
%=============================================================================
\section{Time Discretization}\label{sec:discretization}
This section is devoted to the temporal discretization of the nonlinear system~\eqref{eq:poroNonlinear} by a semi-explicit Euler scheme. We prove first-order convergence under the weak coupling condition of Assumption~\ref{ass:weakCoupling}. 
The proposed approach decouples the system, which means that the two equations can be solved sequentially. At the same time, the system is automatically linearized such that no nonlinear solver is needed, leading to a remarkable boost of efficiency. 

% partition
Throughout this section, we consider an equidistant partition~$0=t_0 < t_1 < \dots < t_N=T$ with step size~$\tau$, i.e., $t_n = n\tau$. 
For simplicity, we assume continuity of the right-hand sides, i.e., $\f\in C([0,T]; \V^*)$, $\g\in C([0,T]; \Q^*)$, and define~$\f^n := \f(t_n) \in\V^*$, $\g^n := \g(t_n) \in\Q^*$. 
Note, however, that point evaluations may also be replaced by integral means if the right-hand sides are only square-integrable. 
Accordingly, we assume continuity of the solution pair~$(u,p)$ and initial data~$p^0\in\Q$. 

% implicit Euler
Before considering the semi-explicit approach, we shortly discuss the standard approach of a fully implicit Euler discretization applied to~\eqref{eq:poroNonlinear}. This leads to the semi-discrete system 
\begin{subequations}
\label{eq:poroImplEuler}
\begin{align}
	a(u^n, v) - d(v, p^n) &\;=\; (\f^n, v), \\
	d(\D u^n,q) + c(\D p^n,q) + b(u^n;p^n,q) &\;=\; (\g^n, q)
\end{align}
\end{subequations}	
for test functions~$v\in\V$, $q\in\Q$. Here, $\D u^n := \tau^{-1}(u^n-u^{n-1})$ denotes the discrete time derivative. 
The solvability of~\eqref{eq:poroImplEuler} for the case $\f^n=0$, $\g^n\in\cHQ$ has been discussed in~\cite[Lem.~2.4]{CaoCM13}, i.e., given~$u^{n-1} \in \calV$ and~$p^{n-1} \in \Q$, there exist~$u^n \in \calV$ and~$p^n \in \calQ$ that solve~\eqref{eq:poroImplEuler}. 

Note that~\eqref{eq:poroImplEuler} still marks a nonlinear system, which needs to be solved in every time step. Thus, a nonlinear solver which comprises an inner iteration is needed. Standard choices implement a Picard iteration; cf.~\cite{CaoCM13,BroV16b,FuCM20}. 
This is a fixed point iteration and has the following form: 
Given~$u^{n-1}\in\calV$, $p^{n-1}\in\calQ$ as approximation at time $t_{n-1}$, we first define~$u^n_0 \coloneqq u^{n-1}$ and $p^n_0 \coloneqq p^{n-1}$. Then, we solve for $j=1,2,\dots$ the \emph{linear} system 
\begin{subequations}
\label{eq:poroImplEuler:Picard}
	\begin{align}
	a(u^n_j, v) - d(v, p^n_j) &\;=\; (\f^n, v), \\
	\tau^{-1} d( u^n_j - u^{n-1}, q) + \tau^{-1} c(p^n_j - p^{n-1}, q) + b(u^n_{j-1}; p^n_j,q) &\;=\; (\g^n,q)
	\end{align}
\end{subequations}	
for all~$v\in\V$, $q\in\Q$. 
Obviously, a fixed point satisfies the nonlinear system~\eqref{eq:poroImplEuler}. 
In practice, one defines $u^n \coloneqq u^n_j$ and $p^n \coloneqq p^n_j$ for some index~$j$, depending on a certain stopping criterion, e.g., the residual of the current iteration. 
For corresponding numerical tests we refer to Section~\ref{ss:sharp}. 
%
%
%=============================================================================
\subsection{Semi-explicit Euler scheme}
We turn to the semi-explicit approach and consider the equidistant partition with step size~$\tau$ as before. 
Given the previous iterates~$u^{n-1} \in \calV$ and~$p^{n-1} \in \Q$, we now aim so solve the system 
\begin{subequations}
\label{eq:poroSemiExpl}
\begin{align}
	a(u^n, v) - d(v, p^{n-1}) &\;=\; (\f^n, v), \label{eq:poroSemiExpl:a} \\
	d(\D u^n,q) + c(\D p^n,q) + b(u^n;p^n,q) &\;=\; (\g^n,q) \label{eq:poroSemiExpl:b}
\end{align}
\end{subequations}	
for all~$v\in\V$, $q\in\Q$. 
In order to be well-posed, we need to discuss the solvability of system~\eqref{eq:poroSemiExpl}, which turns out to be much more straightforward than in the fully implicit case. 
\begin{lemma}[Well-posedness of the semi-explicit scheme~\eqref{eq:poroSemiExpl}]
Consider the setting from Section~\ref{sec:poroNL:linear} with $b$ defined in~\eqref{def:nonlinearB} satisfying Assumption~\ref{ass:b}. 
Further assume~$u^{n-1} \in \calV$, $p^{n-1} \in \Q$, $\f^n \in \V^*$, and~$\g^n\in \Q^*$. 
Then, system~\eqref{eq:poroSemiExpl} attains a unique solution~$u^n\in\calV$ and $p^n \in \calQ$. 
\end{lemma}
\begin{proof}
For fixed~$p^{n-1}$, the term $d(\,\cdot\,,p^{n-1})$ defines a functional in~$\calV^*$. Thus, by the assumptions on the bilinear form~$a$, equation~\eqref{eq:poroSemiExpl:a} provides a unique~$u^n\in\calV$. With this in hand, equation~\eqref{eq:poroSemiExpl:b} can be rewritten as the \emph{linear} variational problem  
\[
 \tilde b(p^n,q) 
 = (\tilde \g^n ,q) 
 := (\g^n,q) + \tau^{-1}\, c(p^{n-1},q) - \tau^{-1}\, d(u^n - u^{n-1},q)	
\]  
with $\tilde b(p^n,q) := b(u^n;p^n,q) + \tau^{-1}\, c(p^n,q)$. 
The right-hand side satisfies~$\tilde \g^n \in \Q^*$ and the form $\tilde b\colon \calQ\times\calQ\to \R$ is elliptic, since
\[
  \tilde b(p,p)
  =  b(u^n;p,p) + \tau^{-1}\, c(p,p) 
  \ge c_b\, \| p\|_\calQ^2.
\]  
Hence, there exists a unique~$p^n \in \calQ$, which completes the proof. 
\end{proof}
In order to prove convergence of the semi-explicit scheme, we follow the idea first presented in~\cite{AltMU21} for the linear case and consider a delay system, which is closely related to the original system~\eqref{eq:poroNonlinear}. 
This is the subject of the following subsection. 
%
%=============================================================================
\subsection{A related delay system}
As an alternative point of view, one may regard the semi-explicit scheme~\eqref{eq:poroSemiExpl} as the \emph{implicit} Euler method applied to the delay system
\begin{subequations}
	\label{eqn:delay:B}
	\begin{align}
	a(\bar{u},v) - d(v, \bar{p}(\,\cdot-\tau)) 
	&= (\f, v), \label{eqn:delay:B:a} \\
	d(\dot {\bar{u}}, q) + c(\dot {\bar{p}},q) + b(\bar{u};\bar{p},q) 
	&= (\g, q) 
	\label{eqn:delay:B:b} 
	\end{align}
\end{subequations}
for test functions~$v\in \V$ and~$q \in \Q$. 
Note that the time delay is exactly the temporal step size $\tau$ and hence fixed. 
For such a delay system, one needs a prescribed \emph{history function} for~$\bar{p}$ in $[-\tau, 0]$ rather than only an initial value. Since there is some freedom of choice regarding the history function, we set~$\bar{p}\big|_{[-\tau, 0]}(t) = \Phi(t)$ and demand 
\begin{equation}
\label{eqn:history:B}
\Phi(-\tau) = \Phi(0) = p^0, \qquad  
\Phi \in C^\infty([-\tau, 0]; \Q).
\end{equation}
Note that one exemplary choice is given by~$\Phi\equiv p^0$. 
In any case, a history function satisfying~\eqref{eqn:history:B} implies~$\bar p(0) = \Phi(0) = p^0$ and by equation~\eqref{eqn:delay:B:a} we conclude~$\bar u(0) = u(0)$, since  
\begin{equation*}
a(\bar u(0),v)
= (\f(0), v) + d(v, \Phi(-\tau)) 
= (\f(0), v) + d(v, p^0).
\end{equation*} 
\begin{remark}[Solvability of the delay system]
Assuming sufficient regularity of the right-hand sides~$f, g$ and the history function~$\Phi$, one can prove the unique solvability of~\eqref{eqn:delay:B}. For this, one can apply Bellmann's method of steps~\cite[Ch.~3.4]{BelZ03}, which considers the intervals $[t_n, t_{n+1}]$ for $n=0, \dots, N-1$ successively. W.l.o.g.~we consider the first interval~$[0,\tau]$, on which we need to solve the parabolic equation 
\[
  c(\dot {\bar{p}},q) + b_0(\bar{p},q) 
  = (\tilde \g, q)
\] 
for all $q\in\Q$ with (using the operator notation from the proof of Proposition~\ref{prop:weakSlnLinearCase}) 
\[
  \tilde \g \coloneqq \g - \calD \calA^{-1} \dot\f - \calD \calA^{-1}\calD^* \dot{\Phi}(\cdot-\tau), \quad
  b_0(p,q) \coloneqq \int_\Omega \frac{\kappa(\nabla\cdot  \calA^{-1}\calD^*\Phi(\cdot-\tau))}{\nu}\, \nabla p \cdot \nabla q \dx.
\]
Note that~$b_0\colon \Q\times\Q\to \R$ is linear, bounded, and (uniformly) elliptic such that~\cite[Th.~26.1]{Wlo92} is applicable. 
\end{remark}
In order to use the interpretation of the semi-explicit discretization as an implicit discretization of the corresponding delay system, we need to show that the two systems~\eqref{eq:poroNonlinear} and~\eqref{eqn:delay:B} only differ by a term of order~$\tau$. To show this, we need the following regularity assumption. 
\begin{assumption}
\label{ass:p}
Given~$p^0\in\Q$ and a history function~$\Phi$ as in~\eqref{eqn:history:B}, let systems~\eqref{eq:poroNonlinear} and~\eqref{eqn:delay:B} be uniquely solvable. Further assume that the solutions are bounded in the sense that 
\[
%  C_p := \|p\|^2_{L^\infty(\Q)} < \infty 
  p \in L^\infty(0,T;\Q)
  \qquad\text{and}\qquad
  \bar p \in W^2(0,T;\cHQ). 
\] 
\end{assumption}
The estimate of the differences $\bar p-p$ and $\bar u-u$ is subject to the following proposition. 
\begin{proposition}[Difference of original and delay system]
\label{prop:delayError}
Within the setting of Section~\ref{sec:poroNL:linear} with $b$ defined in~\eqref{def:nonlinearB} satisfying Assumption~\ref{ass:b}, let~$(u, p)$ denote the solution to~\eqref{eq:poroNonlinear} with $p(0)=p^0\in\Q$ and~$(\bar u, \bar p)$ the solution to~\eqref{eqn:delay:B} for a history function satisfying~\eqref{eqn:history:B}. 
Then, given Assumption~\ref{ass:p}, we have the error bound 
\begin{equation*}
  \|\bar u(t) - u(t) \|^2_{\V} + \|\bar p(t) - p(t) \|^2_{\cHQ} + \int_{0}^{t} \|\bar p(s) - p(s) \|_{\Q}^2\ds 
  \, \leq\,  \frac{\tau^2}{c_0}\, \big(C_1 + C_2\, t\big)\, \exp(C_3\, t), 
\end{equation*}
where $c_0 := \min\{\tfrac{1}{2}c_a,c_b,c_c\}$, $C_1 := 2\, \tfrac{C_d^2}{c_a}\, \|\dot{\bar p}\|_{L^\infty(\cHQ)}^2$, and~$C_2 := C_d\, \|\ddot{\bar p}\|_{L^\infty(\cHQ)}^2$. 
Further, the constant in the exponential term is given by $C_3 := 2\frac{C_d}{c_a} + 2\tfrac{L_b^2}{c_ac_b} \|p\|^2_{L^\infty(\Q)}$.
\end{proposition}
\begin{proof}
Let us introduce the differences~$e_u := \bar u - u$ and~$e_p := \bar p - p$ for which we know that~$e_u(0)=0$ and~$e_p(0)=0$ by the construction of the history function. 
Next, we consider a Taylor expansion of $\bar p$, % and~$\dot{\bar p}$, 
namely
\begin{align}
\label{eqn:taylor}
  \bar p(t-\tau) 
  = \bar p(t) - \tau\, \dot{\bar p}(\xi_t) %, \qquad
 % \dot{\bar p}(t-\tau)   
 % = \dot{\bar p}(t) - \tau\, \ddot{\bar p}(\zeta_t)
\end{align}
for some~$\xi_t \in (t-\tau, t) \subseteq (-\tau, T]$. 
With this, we can derive 
\begin{subequations}
\label{eqn:lemDelayTwoField}
\begin{align}
	a(e_u,v) - d(v, e_p) 
	&= -\tau\, d(v, \dot{\bar p}(\xi_t)), \label{eqn:lemDelayTwoField:a} \\
	d(\dot e_u, q) + c(\dot e_p,q) + b(\bar u;\bar p,q) - b(u;p,q) 
	&= 0 \label{eqn:lemDelayTwoField:b}  
\end{align}
\end{subequations}
for all test functions~$v\in\V$, $q\in\Q$.  
Summing up~\eqref{eqn:lemDelayTwoField:a} and~\eqref{eqn:lemDelayTwoField:b} with test functions $v = \dot e_u$ and $q = e_p$, we obtain
\begin{equation*}
a(e_u,\dot e_u) + c(\dot e_p,e_p) + b(\bar u; \bar p, e_p) - b(u;p,e_p) = -\tau\,d(\dot e_u, \dot{\bar p}(\xi_t))
\end{equation*}
and thus
\begin{equation*}
\tfrac{1}{2}\tddt\|e_u\|_a^2 + \tfrac{1}{2}\tddt\|e_p\|_c^2 + b(\bar u; e_p,e_p) = - \tau\,d(\dot e_u, \dot{\bar p}(\xi_t)) - b(\bar u;p,e_p) + b(u;p,e_p).
\end{equation*}
Integration over~$[0,t]$ and a multiplication by $2$ yields 
\begin{align*}
\|e_u(t)\|_a^2 + \|e_p&(t)\|_c^2 + 2 \int_0^t b(\bar u;e_p,e_p) \ds \\
&= -2\tau \int_0^t \int_{\Omega} \alpha\, (\nabla \cdot \dot e_u) \,\dot{\bar p}(\xi_s)\dx\ds - 2 \int_0^t b(\bar u;p,e_p) - b(u;p,e_p) \ds\\
& = 2\tau \int_0^t\int_{\Omega} \alpha\, (\nabla \cdot e_u) \,\tdds\big(\dot{\bar p}(\xi_s)\big)\dx\ds - 2\tau \int_{\Omega} \alpha\, (\nabla \cdot e_u(t)) \,\dot{\bar p}(\xi_t)\dx\\
&\qquad - 2 \int_0^t b(\bar u;p,e_p) - b(u;p,e_p) \ds,
\end{align*}
where we use integration by parts in the last step. 
For the derivative of~$\dot{\bar p}(\xi_s)$ we now use the fact that 
\[ \tdds \dot{\bar p}(\xi_s) = \tfrac{\dot{\bar p}(s) - \dot{\bar p}(s-\tau)}{\tau} = \ddot{\bar p}(\zeta_s),\]
where we use~\eqref{eqn:taylor} for the first and the existence of an appropriate $\zeta_s \in (s-\tau,s) \subseteq (-\tau, T]$ by the \emph{mean value theorem} for the second equality. 
By the ellipticity of~$b$ we conclude that  
\begin{align*}
\|e_u&(t)\|_a^2 + \|e_p(t)\|_c^2 + 2 \,c_b \int_0^t \|e_p\|^2_\Q \ds \notag \\
&\quad\le 2\tau \int_0^t d(e_u, \ddot{\bar p}(\zeta_s)) \ds - 2\tau\, d(e_u(t), \dot{\bar p}(\xi_t)) - 2 \int_0^t b(\bar u;p,e_p) - b(u;p,e_p) \ds. %\label{eqn:proof:delayError:1}
\end{align*}
By Assumption~\ref{ass:b} we have
\[
  b(\bar u;p,e_p) - b(u;p, e_p)
  \le L_b\, \|p\|_\Q \|e_u\|_\V \|e_p\|_\Q
\]
and thus, 
% Rechnug mit Young:
%\begin{align*}
%%...\|e_u&(t)\|_a^2 + \|e_p(t)\|_c^2 + 2 \,c_b \int_0^t \|e_p\|^2_\Q \ds \notag \\
%&\le 2\tau \int_0^t d(e_u, \ddot{\bar p}(\zeta_s)) \ds - 2\tau\, d(e_u(t), \dot{\bar p}(\xi_t)) - 2 \int_0^t b(\bar u;p,e_p) - b(u;p,e_p) \ds \\
%%
%&\le 
%2\tau \int_0^t C_d \|e_u\|_\V \|\ddot{\bar p}(\zeta_s)\|_{\cHQ} \ds 
%+2\tau C_d\, \|e_u(t)\|_\V \|\dot{\bar p}(\xi_t)\|_{\cHQ} 
%+2L_b\, \int_0^t \|p\|_\Q \|e_u\|_\V \|e_p\|_\Q \ds\\
%%
%&\le 
%C_d \int_0^t  \|e_u\|^2_\V + \tau^2\|\ddot{\bar p}(\zeta_s)\|^2_{\cHQ} \ds 
%+ \tfrac{c_a}{2} \|e_u(t)\|^2_\V + \tfrac{2}{c_a} C_d^2 \tau^2\|\dot{\bar p}\|_{L^\infty(\cHQ)}^2 \\
%&\qquad+ 2L_b\, \|p\|_{L^\infty(\Q)} \int_0^t  \|e_u\|_\V \|e_p\|_\Q \ds\\
%%
%&\le 
%C_d\tau^2t\, \|\ddot{\bar p}\|^2_{L^\infty(\cHQ)} + C_d \int_0^t  \|e_u\|^2_\V \ds 
%+ \tfrac{1}{2} \|e_u(t)\|^2_a + 2\tfrac{C_d^2}{c_a} \tau^2\|\dot{\bar p}\|_{L^\infty(\cHQ)}^2  \\
%&\qquad+ \tfrac{L_b^2}{c_b} \|p\|^2_{L^\infty(\Q)} \int_0^t  \|e_u\|^2_\V \ds + c_b \int_0^t \|e_p\|^2_\Q \ds
%\end{align*}
%
multiple applications of the weighed Young's inequality~\cite[App.~B]{Eva98} yield
\begin{align*}
\|e_u(t)\|_a^2 + \|e_p&(t)\|_c^2 + 2 \,c_b \int_0^t \|e_p\|^2_\Q \ds\\
&\leq C_d\tau^2 t\,\|\ddot{\bar p}\|_{L^\infty(\cHQ)}^2  + C_d\int_0^t \|e_u\|^2_\V\ds + 2 \tfrac{C_d^2}{c_a}\tau^2\, \|\dot{\bar p}\|_{L^\infty(\cHQ)}^2\\
&\qquad\quad + \tfrac{1}{2}\|e_u(t)\|^2_a
+ c_b\int_0^t  \|e_p\|_\Q^2 \ds + \tfrac{L_b^2}{c_b} \|p\|^2_{L^\infty(\Q)} \int_0^t\|e_u\|^2_\V\ds.
\end{align*}
We can now absorb the first two terms in the last line to get
\begin{align*}
\tfrac 12 \|e_u(t)\|_a^2 + \|e_p(t)\|_c^2 + c_b \int_0^t \|e_p\|^2_\Q \ds
%
%&\le C_d\tau^2 t\,\|\ddot{\bar p}\|_{L^\infty(\cHQ)}^2  + C_d\int_0^t \|e_u\|^2_\V\ds + 2 \tfrac{C_d^2}{c_a}\tau^2\, \|\dot{\bar p}\|_{L^\infty(\cHQ)}^2 + \tfrac{L_b^2}{c_b} \|p\|^2_{L^\infty(\Q)} \int_0^t\|e_u\|^2_\V\ds \\
%
\le \tau^2\, (C_1 + C_2\, t) + C_3 \int_0^t \tfrac 12\, \|e_u\|^2_a \ds
\end{align*}
with constants~$C_1 := 2 \tfrac{C_d^2}{c_a}\, \|\dot{\bar p}\|_{L^\infty(\cHQ)}^2$, $C_2 := C_d\, \|\ddot{\bar p}\|_{L^\infty(\cHQ)}^2$, and~$C_3 := 2\frac{C_d}{c_a} + 2\tfrac{L_b^2}{c_ac_b} \|p\|^2_{L^\infty(\Q)}$. 
Hence, an application of Gr\"onwall's inequality yields  
\begin{align*}
  \tfrac 12 \|e_u(t)\|_a^2 + \|e_p(t)\|_c^2 + c_b \int_0^t \|e_p\|^2_\Q \ds
  \le \tau^2\, (C_1 + C_2\, t)\, \exp( t\, C_3). 
\end{align*}
The assertion follows with the lower bounds of the bilinear forms.  
\end{proof}
%
%
%=============================================================================
\subsection{Proof of convergence}
After we have seen that the pairs~$(u,p)$ and~$(\bar u,\bar p)$ only differ by a term of order~$\tau$, we now analyze the error caused by the \emph{implicit} discretization of the delay system~\eqref{eqn:delay:B}. 
For this, we assume that the weak coupling condition introduced in Section~\ref{sec:poroNL:assumptions} holds. 
\begin{proposition}[Semi-discrete error for the delay system]
\label{prop:discError:Delay}
Consider once more the assumptions of Proposition~\ref{prop:delayError} including Assumptions~\ref{ass:b} and~\ref{ass:p}. Further consider the weak coupling condition from Assumption~\ref{ass:weakCoupling} as well as  
\begin{equation}
\label{eq:boundTau}
  \tau < \frac{c_a c_b}{2\, L_b^2\, \|\bar p\|^2_{L^\infty(\Q)} }.
\end{equation}
Let~$(\bar u, \bar p)$ denote the exact solution to~\eqref{eqn:delay:B} and~$(u^n,p^n)$ the sequence resulting from~\eqref{eq:poroSemiExpl} for $n \le T/\tau$ and exact initial data. % dh $u^0=\bar u(0), p^0 = \bar p(0)$
Then we have the error bounds  
\begin{equation*}
  \|\bar u(t_n) - u^n \|^2_{\V} + \|\bar p(t_n) - p^n \|^2_{\cHQ} 
  \ \le\ \tau^2\, \frac{C_1}{C_2}\, \big(\! \exp(C_2\, t_n) - 1 \big)
\end{equation*}
and 
\begin{equation*}
  \sum_{k=1}^{n} \tau\, \|\bar p(t_k) - p^k \|_{\Q}^2 
  \ \le\ 
  \tau^2\, C_1\, t_n\, \exp(C_2\, t_n),  
\end{equation*}
where $C_1 := \tfrac{4}{c_b}\, \tilde C_d^2\, \|\ddot{\bar u}\|_{L^\infty(\cHV)} 
+ \tfrac{4}{c_b}\, C_c^2\, \CQtoHsquare\, \|\ddot{\bar p}\|_{L^\infty(\cHQ\!)}$ and $C_2 := \tfrac{2 L^2_b}{c_a c_b} \|\bar p\|_{L^\infty(\Q)}^2$. 
\end{proposition}
\begin{remark}\label{rem:tau}
The step size restriction~\eqref{eq:boundTau} is solely dependent on the problem at hand. In particular, this condition is \emph{not} a CFL-type condition that would couple $\tau$ to underlying spatial discretization parameters.
Moreover, assuming sufficiently smooth data and a so-called \emph{splicing condition}, it can be shown that~$\bar p$ is bounded independently of~$\tau$; cf.~\cite[App.]{AltMU21}. Hence, \eqref{eq:boundTau} displays a well-defined condition. 
\end{remark}
\begin{proof}[Proof of \Cref{prop:discError:Delay}] 
The proof is based on~\cite{AltMU21} and follows the ideas of~\cite{ErnM09}. We set 
\begin{displaymath}
\eta_u^n \vcentcolon= \bar u^n - u^n \in \V\qquad\text{and}\qquad
\eta_p^n \vcentcolon= \bar  p^n - p^n \in \Q 
\end{displaymath}
as well as
\begin{displaymath}
\theta^{n+1}_u \vcentcolon= \bar u^{n+1} - \bar u^n - \tau \dot{\bar u}^{n+1}\in \V \qquad\text{and}\qquad
\theta^{n+1}_p \vcentcolon= \bar p^{n+1} -\bar p^n -\tau \dot{\bar p}^{n+1}\in \Q, 
\end{displaymath}
where $\bar u^{n}\vcentcolon=\bar u(t_{n})$ and $\bar p^{n} \vcentcolon=\bar p(t_{n})$ are the (pointwise) solutions of~\eqref{eqn:delay:B} and $(u^n, p^n)$ the discrete solution of~\eqref{eq:poroSemiExpl} at time point $t_n = \tau n$. 
By the assumption on the initial data, we have~$\eta_u^0=0$ and~$\eta_p^0=0$. 
Using~\eqref{eq:poroSemiExpl:a} and~\eqref{eqn:delay:B:a}, we immediately obtain
\begin{equation}\label{eqn:discError:Delay:B:proof1}
\begin{aligned}
a(\eta^{n+1}_u,v) - d(v,\eta^{n+1}_p) 
%	&= a(\Ru \bar u^{n+1} - u^{n+1}_h,v_h) - d(v_h,\eta^{n}_p) - d(v_h,\eta^{n+1}_p-\eta^{n}_p)\\
&= a(\bar u^{n+1}-u^{n+1},v) - d(v,\bar p^{n}-p^{n}) - d(v,\eta^{n+1}_p- \eta^{n}_p)\\
&= - d(v,\eta^{n+1}_p- \eta^{n}_p)
\end{aligned}
\end{equation}
for all $v\in \V$. 
Further, it holds that
\begin{equation}\label{eqn:discError:Delay:B:proof2}
\begin{aligned}
\tau\, b(u^{n+1};\eta^{n+1}_p,q) 
= \tau\, b(\bar u^{n+1}&;\bar p^{n+1},q) - \tau\, b(u^{n+1};p^{n+1},q) 
\\ &\quad+ \tau\big[b(u^{n+1};\bar p^{n+1},q) - b(\bar u^{n+1};\bar p^{n+1},q)\big]
\end{aligned}
\end{equation}
for all $q \in \Q$.
With~\eqref{eq:poroSemiExpl:b}, \eqref{eqn:delay:B:b}, and~\eqref{eqn:discError:Delay:B:proof2}, we have that 
\begin{align}
%\begin{aligned}
d(&\eta^{n+1}_u - \eta^{n}_u,q) + c(\eta^{n+1}_p - \eta^n_p,q) + \tau\, b(u^{n+1};\eta^{n+1}_p,q) \nonumber\\
&\ = d(\bar u^{n+1} - \bar u^n - \tau D_\tau u^{n+1},q) + c(\bar p^{n+1} - \bar p^n - \tau D_\tau p^{n+1},q) + \tau\, b(u^{n+1};\eta^{n+1}_p,q)\label{eqn:discError:Delay:B:proof3}\\
%
%&\ = d(\theta^{n+1}_u + \tau \dot{\bar u}^{n+1} - \tau D_\tau u^{n+1},q) + c(\theta^{n+1}_p + \tau \dot{\bar p}^{n+1} - \tau D_\tau p^{n+1},q) + \tau\, b(u^{n+1};\eta^{n+1}_p,q)\\
%
%&\ = d(\theta^{n+1}_u,q) + c(\theta^{n+1}_p,q) - \tau\, b(\bar{u}^{n+1};\bar{p}^{n+1},q) + \tau\, b(u^{n+1};p^{n+1},q)
%+ \tau\, b(u^{n+1};\eta^{n+1}_p,q)\\
%
&\ = d(\theta^{n+1}_u,q) + c(\theta^{n+1}_p,q) + \tau\, \big[b(u^{n+1};\bar p^{n+1},q) - b(\bar u^{n+1};\bar p^{n+1},q)\big]\nonumber
%\end{aligned}
\end{align}
for all $q\in \Q$. 
Summing up \eqref{eqn:discError:Delay:B:proof1} and \eqref{eqn:discError:Delay:B:proof3} for the particular choices~$v = \eta^{n+1}_u - \eta^n_u$ and $q = \eta^{n+1}_p$, we obtain 
\begin{equation}\label{eqn:discError:Delay:B:proof4}
\begin{aligned}
a(\eta^{n+1}_u,\eta^{n+1}_u& - \eta^n_u) + c(\eta^{n+1}_p-\eta^n_p,\eta^{n+1}_p) + \tau\, b(u^{n+1};\eta^{n+1}_p,\eta^{n+1}_p)\\ 
&= -d(\eta^{n+1}_u - \eta^n_u,\eta^{n+1}_p - \eta^{n}_p) + d(\theta^{n+1}_u,\eta^{n+1}_p) + c(\theta^{n+1}_p,\eta^{n+1}_p)\\
&\hspace{3cm}+ \tau\, \big[b(u^{n+1};\bar p^{n+1},\eta^{n+1}_p) - b(\bar u^{n+1};\bar p^{n+1},\eta^{n+1}_p)\big].
\end{aligned}
\end{equation}
In the following, we apply the identity~$2\, a(u, u-v) = \Vert u \Vert^2_a - \Vert v \Vert^2_a + \Vert u- v \Vert^2_a$ and the corresponding formula for the bilinear form~$c$. Hence, with~\eqref{eqn:discError:Delay:B:proof4} and Assumption~\ref{ass:b} we further get 
\begin{align}
\|\eta^{n+1}_u\|^2_a - \|\eta^n_u\|^2_a& + \|\tau D_\tau\eta^{n+1}_u\|^2_a + \|\eta^{n+1}_p\|^2_c - \|\eta^n_p\|^2_c + \|\tau D_\tau\eta^{n+1}_p\|^2_c + 2c_b\tau\, \|\eta^{n+1}_p\|^2_\Q \notag  \\
&\le -2\, d(\tau D_\tau\eta^{n+1}_u, \tau D_\tau\eta^{n+1}_p) + 2\, d(\theta^{n+1}_u,\eta^{n+1}_p) + 2\, c(\theta^{n+1}_p,\eta^{n+1}_p) \label{eqn:discError:Delay:B:proof5} \\
&\hspace{2.5cm}+ 2\tau\,\big[b(u^{n+1};\bar p^{n+1},\eta^{n+1}_p) - b(\bar u^{n+1};\bar p^{n+1},\eta^{n+1}_p)\big]. \notag 
\end{align}
Next, we consider a weighted version of Young's inequality, which gives 
\begin{equation}\label{eqn:discError:Delay:B:proof6}
-2\, d(\tau D_\tau \eta^{n+1}_u,\tau D_\tau\eta^{n+1}_{p}) \leq \tfrac{C_{d}^2}{c_a\,c_c}\|\tau D_\tau \eta^{n+1}_u\|_a^2 + \|\tau D_\tau\eta^{n+1}_{p}\|_c^2.
\end{equation}
For the second and third term on the right-hand side of \eqref{eqn:discError:Delay:B:proof5} we similarly obtain 
\begin{equation}\label{eqn:discError:Delay:B:proof7}
2\, d(\theta^{n+1}_u,\eta^{n+1}_p) 
\leq 2\, \tilde C_d\, \|\theta^{n+1}_u\|_{\cHV} \,\|\eta^{n+1}_p\|_{\Q}
\leq \tfrac{\tilde C_d^2}{c_b}\tfrac{4}{\tau}\|\theta^{n+1}_u\|_{\cHV}^2 + \tfrac{c_b\tau}{4} \|\eta^{n+1}_p\|^2_{\Q}
\end{equation}
and with the continuity constant~$\CQtoH$ of the embedding~$\Q\hook \cHQ$, 
\begin{equation}\label{eqn:discError:Delay:B:proof8}
2\, c(\theta^{n+1}_p,\eta^{n+1}_p) 
\leq 2\,C_c \|\theta^{n+1}_p\|_\cHQ \, \|\eta^{n+1}_p\|_\cHQ 
\leq \tfrac{C_c^2\, \CQtoHsquare}{c_b}\tfrac{4}{\tau} \|\theta^{n+1}_p\|^2_\cHQ + \tfrac{c_b\tau}{4} \|\eta^{n+1}_p\|^2_\Q. 
\end{equation}
For the last term of the right-hand side in~\eqref{eqn:discError:Delay:B:proof5}, we apply once more the weighted version of Young's inequality. This leads to the estimate 
\begin{align}\label{eqn:discError:Delay:B:proof9}
2\tau\,\big[b(u^{n+1};\bar p^{n+1},\eta^{n+1}_p) - b(\bar u^{n+1};\bar p^{n+1},\eta^{n+1}_p)\big] 
&\le 2\tau L_b \|\bar p^{n+1}\|_\Q \|\eta^{n+1}_p \|_\Q \| \eta_u^{n+1} \|_\V \notag \\
%&\leq \tfrac{\tau}{\gamma}\|\eta_u^{n+1}\|_a^2 + \tau\tfrac{\gamma L^2_b}{c_a} \|\bar p^{n+1}\|_\Q^2\,\|\eta_p^{n+1}\|^2_\Q \\
&\leq \tfrac{c_b\tau}{2} \|\eta_p^{n+1}\|^2_\Q + \tfrac{2\tau L^2_b}{c_a c_b} \|\bar p^{n+1}\|_\Q^2\, \|\eta_u^{n+1}\|_a^2. 
\end{align}
We now combine the estimates~\eqref{eqn:discError:Delay:B:proof5}--\eqref{eqn:discError:Delay:B:proof9} and absorb the terms $\|\tau D_\tau\eta^{n+1}_p\|_c^2$, $c_b\tau\, \|\eta^{n+1}_p\|_{\Q}$, and $\|\tau D_\tau\eta^{n+1}_u\|^2_a$. For the latter, we use \Cref{ass:weakCoupling}. 
In total, this yields
\begin{align*}
\|\eta^{n+1}_u\|^2_a - \|\eta^n_u\|^2_a& + \|\eta^{n+1}_p\|^2_c - \|\eta^n_p\|^2_c + c_b\tau\, \| \eta^{n+1}_p\|^2_\Q \notag \\
&\leq \tau\, \tfrac{2 L^2_b}{c_a c_b} \|\bar p^{n+1}\|_\Q^2\, \|\eta_u^{n+1}\|_a^2
+ \tfrac{4}{c_b \tau}\, \Big( \tilde C_d^2\, \|\theta^{n+1}_u\|_{\cHV}^2 
+ C_c^2\, \CQtoHsquare\, \|\theta^{n+1}_p\|^2_\cHQ \Big).  
\end{align*}
Applying the \emph{mean value theorem}, we estimate
\begin{equation*}
\|\theta^{n+1}_u\|_{\cHV} 
\leq\tau^2\, \|\ddot{\bar u}\|_{L^\infty(\cHV)}
\qquad \text{and}\qquad \|\theta^{n+1}_p\|_\cHQ 
\leq \tau^2\, \|\ddot{\bar p}\|_{L^\infty(\cHQ\!)}.
\end{equation*}
Hence, we obtain 
\begin{align}
\label{eqn:estimateEtaAtN}
  \|\eta^{n+1}_u\|^2_a - \|\eta^n_u\|^2_a& + \|\eta^{n+1}_p\|^2_c - \|\eta^n_p\|^2_c + c_b\tau\, \| \eta^{n+1}_p\|^2_\Q 
  \le \tau\, C_2\, \|\eta_u^{n+1}\|_a^2 + \tau^3\, C_1  
\end{align}
with constants
\[
  C_1 := \tfrac{4}{c_b}\, \tilde C_d^2\, \|\ddot{\bar u}\|_{L^\infty(\cHV)} 
  + \tfrac{4}{c_b}\, C_c^2\, \CQtoHsquare\, \|\ddot{\bar p}\|_{L^\infty(\cHQ\!)}
  \qquad\text{and}\qquad
  C_2 := \tfrac{2 L^2_b}{c_a c_b} \|\bar p\|_{L^\infty(\Q)}^2.
\]
Due to the assumption on the step size~\eqref{eq:boundTau}, which now reads~$\tau\, C_2 < 1$, we can apply a discrete Gr\"onwall inequality to~\eqref{eqn:estimateEtaAtN}, see~\cite[Prop.~3.1]{Emm99}. This yields 
\begin{align*}
  \|\eta^{n}_u\|^2_a + \|\eta^{n}_p\|^2_c 
  \le \tau^2\, \frac{C_1}{C_2} \big(\! \exp(C_2 t_n) - 1 \big).
%  \le \tau^2\, \frac{C_1}{C_2}\, \exp(C_2 t_n).  
\end{align*}
On the other hand, the summation of~\eqref{eqn:estimateEtaAtN} over~$n$ (using $\eta_u^0=0$ and~$\eta_p^0=0$) and using the previous estimate yields
\begin{align}
%  \|\eta^{n}_u\|^2_a + \|\eta^{n}_p\|^2_c + 
  c_b\tau\, \sum^n_{k=1} \| \eta^{k}_p\|^2_\Q 
  \le \tau\, C_2\, \sum^n_{k=1} \|\eta_u^{k}\|_a^2 + \tau^2\, C_1\, t_n 
%  \le \tau^3\, C_1\, \sum^n_{k=1} \big( \exp(C_2 t_n) - 1 \big) + \tau^2\, t_n\, C_1 
%  = \tau^2\, t_n\, C_1\, \big( \exp(C_2 t_n) - 1 \big) + \tau^2\, t_n\, C_1
  \le \tau^2\, C_1\, t_n\, \exp(C_2 t_n)
\end{align}
and hence the assertion.  
\end{proof}
\begin{remark}
In the case of non-homogeneous boundary conditions, estimate~\eqref{eqn:discError:Delay:B:proof7} needs to be adjusted. More precisely, $\theta_u^{n+1}$ needs to be measured in the $\V$-norm, leading to a constant $C_1$, which depends on~$\|\ddot{\bar u}\|_{L^\infty(\mathcal V)}$ rather than $\|\ddot{\bar u}\|_{L^\infty(\mathcal H_{\mathcal V})}$. 
\end{remark}
%
%
%=============================================================================
\subsection{Summary and main result}
For the main result, which states first-order convergence of the semi-explicit method for poroelasticity with nonlinear permeability, we combine the two previous propositions. This then leads to the following statement. 
\begin{theorem}[Convergence of the semi-explicit scheme]
\label{thm:convergence}
Consider the setting of Section~\ref{sec:poroNL:linear} with $b$ defined in~\eqref{def:nonlinearB} as well as Assumptions~\ref{ass:weakCoupling}, \ref{ass:b}, and~\ref{ass:p}. 
Let~$(u, p)$ denote the exact solution to~\eqref{eq:poroNonlinear} with $p(0)=p^0\in\Q$ and~$(u^n,p^n)$ the sequence resulting from~\eqref{eq:poroSemiExpl} for $n \le T/\tau$ with the same initial data. 
If the step size satisfies the restriction~\eqref{eq:boundTau}, then there exists a constant $C$ (depending on $p$ and $\bar p$) such that
%$C := 2\frac{C_d}{c_a} + 2\tfrac{L_b^2}{c_ac_b} \|p\|^2_{L^\infty(\Q)}$.
%$C := \tfrac{2 L^2_b}{c_a c_b} \|\bar p\|_{L^\infty(\Q)}^2$.
%
\[
  \|u(t_n) - u^n \|^2_{\V} + \|p(t_n) - p^n \|^2_{\cHQ}
  \ \lesssim\ \tau^2\, (1+t_n)\, \exp(C\,t_n).
\]
\end{theorem}
\begin{proof}
One can easily define a history function satisfying~\eqref{eqn:history:B}. 
The assertion directly follows from the two previous Propositions~\ref{prop:delayError} and~\ref{prop:discError:Delay} and the triangle inequality. 
\end{proof}
Recall that we have no pointwise estimates of the pressure variable in the $\Q$-norm. To get an estimate in~$L^2(0,T;\Q)$, let us define the piecewise constant function~$P\in L^2(0,T;\Q)$ by
\[
  P(0) := p^0, \qquad
  P(t) := p^k \text{ for } t\in(t_{k-1}, t_k]. 
\]  
Then the combination of Propositions~\ref{prop:delayError} and~\ref{prop:discError:Delay} shows 
\begin{equation*}
  \| p - P \|_{L^2(0,t_n;\Q)}^2 
  = \int_{0}^{t_n} \| p(t) - P(t) \|_{\Q}^2 \dt
  \lesssim \tau^2\, (1+t_n)\, \exp(C\,t_n).
\end{equation*}
%
%% proof of this estimate:
%\begin{proof}
%By triangle inequality we get	
%%
%\begin{equation*}
%	\int_{0}^{t_n} \| p(t) - P(t) \|_{\Q}^2 \dt
%	\le 2 \int_{0}^{t_n} \| p(t) - \bar p(t) \|_{\Q}^2 \dt + 2 \int_{0}^{t_n} \| \bar p(t) - P(t) \|_{\Q}^2 \dt
%\end{equation*}
%%
%where the first term is bounded by $\tau^2\, (1+t_n)\, \exp(C\,t_n)$ using \Cref{prop:delayError}. For the second term we apply \Cref{prop:discError:Delay} and a simple quadrature rule, leading to 
%%
%\begin{align*}
%  \int_{0}^{t_n} \| \bar p(t) - P(t) \|_{\Q}^2 \dt
%  &\le 2\sum_{k=1}^{n} \Big( \int_{t_{k-1}}^{t_k} \| \bar p(t_k) - p^k \|_{\Q}^2 \dt + \int_{t_{k-1}}^{t_k} \| \bar p(t) - \bar p(t_k) \|_{\Q}^2 \dt \Big) \\
%  &\lesssim \tau^2\, t_n\, \exp(C\, t_n) + \tau^2\, t_n\, \| \bar p \|_{W^{1,\infty}(\Q)}^2.
%\end{align*}
%\end{proof}
%
\begin{remark}[spatial discretization]
For practical computations, one would consider discrete spaces $V_h\subseteq \V$ and $Q_h\subseteq \Q$, leading to approximations~$u^n_h \approx u(t_n)$ and~$p^n_h \approx p(t_n)$. Then, a similar convergence result can be shown based on spatial projections corresponding to $a$ and $b$, namely $\Ru \colon \V\to  V_h$ and (for a fixed $w\in\V$) $\Rp^w \colon \Q \to Q_h$, defined by 
\[
  a(\Ru u,v_h) = a(u,v_h), \qquad
  b(w; \Rp^w p, q_h) = b(w;p,q_h)
\]
for all $v_h \in V_h$, $q_{h} \in Q_{h}$. 
Assuming approximation properties  
\[
  \|u - \Ru u\|_\cHV \lesssim h\,\|u\|_\V, \quad
  \|u - \Ru u\|_\V \lesssim h\,\|\nabla^2 u\|_\cHV
\]
and 
\[
  \|p - \Rp^w p\|_\cHQ \lesssim h\,\|p\|_\Q, \quad
  \|p - \Rp^w p\|_\Q \lesssim h\,\|\nabla^2 p\|_\cHQ,\quad
  \|p - \Rp^w p\|_\cHQ \lesssim h^2\,\|\nabla^2 p\|_\cHQ, 
\]
uniformly in $w$, we obtain an error estimate of the form 
\begin{align*}
  \| u(t_n)- u^n_h\|^2_a + \| p(t_n) - p^n_h\|^2_c  
  %+ c_b\tau\, \| \eta^{n+1}_p\|^2_\Q 
  \lesssim  (\tau^2 + h^2 + h^4/\tau)\, \exp(C\, t_n).
\end{align*} 
Hence, in the (not critical) regime $h^2\lesssim \tau$ we have first-order convergence in $\tau$ and $h$. 
For the convergence proof one considers the (discrete) error terms~$\eta_u^n \vcentcolon= \Ru \bar u^n - u^n_h \in V_h$ and $\eta_p^n \vcentcolon= \Rp^{u^n_h}\bar p^n - p^n_h \in Q_h$. Then, calculations similar to the ones presented in the proof of~\Cref{prop:discError:Delay} combined with the assumed approximation properties of~$\Ru$ and~$\Rp^w$ as well as a step size restriction in the spirit of~\eqref{eq:boundTau} yield the result. 
\end{remark}
We now turn to the numerical investigation of the semi-explicit scheme, considering three test cases. 
%
%
%
%=============================================================================
%=========  Numerics
%=============================================================================
\section{Numerical Examples}\label{sec:num}
Besides the numerical validation of the obtained convergence rates, this section is devoted to the following questions: 
\begin{itemize}
	\item performance of the semi-explicit scheme compared to implicit schemes, 
	\item necessity of the weak coupling condition given in Assumption~\ref{ass:weakCoupling}.
\end{itemize}
All computations are based on a finite element implementation in Python based on the computing platform \emph{FEniCS}. Throughout this section, $h$ denotes the spatial discretization parameter that corresponds to a classical first-order finite element approximation on a regular mesh. Moreover, we use equidistant time steps as mentioned in Section~\ref{sec:discretization}. 
%
%
%=============================================================================
\subsection{Network-inspired model}\label{ss:network} 
In this first example, we consider the network-inspired porosity-permeability relation as presented in Example~\ref{exp:networkInspiredPermeability} with
\begin{equation*}
\rho_0 = 0.4,\quad \hat\rho = 0.2,\quad 
\delta = 0.01, \quad T = 1.
\end{equation*}
As computational domain we use the unit square $D=(0,1)^2$. The remaining parameters are based on the values for \emph{Boise sandstone} (see~\cite[Sect.~3.3.4]{DetC93}) and are given by 
\begin{equation*}
\lambda = 7.826\cdot10^{8}, \quad 
\mu = 1.826\cdot 10^{9}, \quad 
\alpha = 0.85, \quad 
M = 7\cdot 10^{9},\quad 
\kappa_0/\nu = 8\cdot 10^{-10}.
\end{equation*}
The right-hand sides and the initial condition are given by $f \equiv 0$,
\begin{align*}
g(x,t) &= 30\,\sin(\pi x_1)\,\exp(-t),\\ 
p^0(x) &= 50\,(1-x_1)\, x_1\, (1-x_2)\, x_2.
\end{align*}
We emphasize that the coupling condition from~\Cref{ass:weakCoupling} is slightly violated for this setup, since $C_d^2 = \alpha^2 = 0.725$ and $c_a\,c_c = \mu/M = 0.261$. This, however, is not critical for the stability in this example. 

To investigate the performance of the semi-explicit scheme, we compare the results for multiple mesh sizes~$h$ and time step sizes~$\tau$ with a reference solution computed with $h_\mathrm{ref} = 2^{-7}$ and $\tau_\mathrm{ref} = 2^{-8}$ and the implicit scheme \eqref{eq:poroImplEuler:Picard}, where the inner iteration is handled with a Picard-type approach. The inner iteration stops when a relative residual error of $10^{-9}$ is reached. For this model, we have observed that at most seven Picard steps were needed to converge for each point in time.

For a numerical comparison, we also present the approximations for different discretization parameters with an implicit approach coupled with an inner Picard iteration as for the reference solution. The stopping criterion for the inner iteration is a relative residual error of $10^{-9}$ and at most nine Picard steps were needed to reach this threshold. In Figure~\ref{fig:ex1_t}, we present the errors at the final point in time $t=T=t_N$ measured in the norms $\|\cdot\|_a$ (equivalent to the $\V$-norm) and $\|\cdot\|_c$ (equivalent to the $\cHQ$-norm) for the variables $u$ and $p$, respectively. 
The mesh size $h = 2^{-7}$ is chosen such that the temporal error dominates. Based on the above theory, we expect linear convergence in $\tau$ for the semi-explicit scheme. This is observed for both variables $p$ and $u$. Note that the errors of the implicit scheme combined with a Picard iteration show the same convergence rate but are smaller by a factor 10 for $p$ and 100 for $u$. 
We would like to emphasize, however, that the semi-explicit scheme is very fast compared to the implicit one, which outweighs the higher error; cf.~also the comparison in run times in Section~\ref{ss:KozenyCarman}. Let us also mention that the errors in the $\cHV$-norm (for $u$) and the $\Q$-norm (for $p$) show a very similar linear behavior in $\tau$ as well. 

\begin{figure}
	\scalebox{0.85}{
		% This file was created by tikzplotlib v0.8.2.
\begin{tikzpicture}

\begin{axis}[
legend cell align={left},
legend columns = 3,
legend style={at={(0.35,1.15)}, anchor=north west, draw=white!80.0!black},
log basis x={10},
log basis y={10},
tick align=outside,
tick pos=both,
x grid style={white!69.01960784313725!black},
xlabel={time step \(\displaystyle \tau\)},
xmin=0.0063457218465331, xmax=0.615572206672458,
xmode=log,
xtick style={color=black},
y grid style={white!69.01960784313725!black},
ylabel={$\|p(T)-p_h^N\|_{c}\  /\ \|p(T)\|_{c}$},
yticklabels = {},
y label style={at={(.07,0.5)}},
ymin=0.000019, ymax=0.85,
ymode=log,
ytick style={color=black}
]
\addplot [very thick, color0!70!white, mark=o, mark size=4.5, mark options={solid}]
table {%
0.5 0.021603010749598
0.25 0.0152839761558142
0.125 0.00716181792478684
0.0625 0.00340437658944168
0.03125 0.00165177739681034
0.015625 0.000804993360684104
0.0078125 0.000388788942122177
};
\addlegendentry{semi-explicit Euler$\qquad$}

\addplot [very thick, color3, dashed, mark=triangle*, mark size=3.5, mark options={solid}]
table {%
0.5 0.00247722795754455
0.25 0.00188545230067744
0.125 0.00088710289593823
0.0625 0.000418217298119111
0.03125 0.000192690743536223
0.015625 8.20607575064487e-05
0.0078125 2.72675466662299e-05
};
\addlegendentry{implicit Euler$\qquad$}

\addplot [very thick, gray, dashed]
table {%
	0.5 0.16
	0.005 0.0016
};
\addlegendentry{order 1}

\end{axis}

\end{tikzpicture}
		\hspace{-6.5cm}
		% This file was created by tikzplotlib v0.8.2.
\begin{tikzpicture}

\begin{axis}[
%legend cell align={left},
%legend style={at={(0.03,0.97)}, anchor=north west, draw=white!80.0!black},
log basis x={10},
log basis y={10},
tick align=outside,
tick pos=both,
x grid style={white!69.01960784313725!black},
xlabel={time step \(\displaystyle \tau\)},
xmin=0.0063457218465331, xmax=0.615572206672458,
xmode=log,
xtick style={color=black},
y grid style={white!69.01960784313725!black},
ylabel={$\|u(T)-u_h^N\|_{a}\  /\ \|u(T)\|_{a}$},
y label style={at={(1.28,0.5)}},
ymin=0.000019, ymax=0.85,
ymode=log,
ytick style={color=black}
]

\addplot [very thick, color0!70!white, mark=o, mark size=4.5, mark options={solid}]
table {%
0.5 0.599505326310726
0.25 0.290668312035225
0.125 0.136471455912541
0.0625 0.0658750741454815
0.03125 0.0323616022578886
0.015625 0.0160293904546532
0.0078125 0.00796667346145882
};
%\addlegendentry{semi-expl}

\addplot [very thick, color3, dashed, mark=triangle*, mark size=3.5, mark options={solid}]
table {%
0.5 0.00251787836692146
0.25 0.00195814603419357
0.125 0.000922835289384071
0.0625 0.000435334912561251
0.03125 0.000200637287592462
0.015625 8.54573975654814e-05
0.0078125 2.83982693905334e-05
};
%\addlegendentry{impl}

\addplot [very thick, gray, dashed]
table {%
	0.5 0.16
	0.005 0.0016
};
\end{axis}

\end{tikzpicture}
	}
	\caption{Relative error in $p$ (left, measured in the $c$-norm) and $u$ (right, measured in the $a$-norm) in the example of Section~\ref{ss:network} at the final time $T$ for fixed $h = 2^{-7}$ and varying $\tau$. %The dashed line indicates order 1.
	}
	\label{fig:ex1_t}
\end{figure}
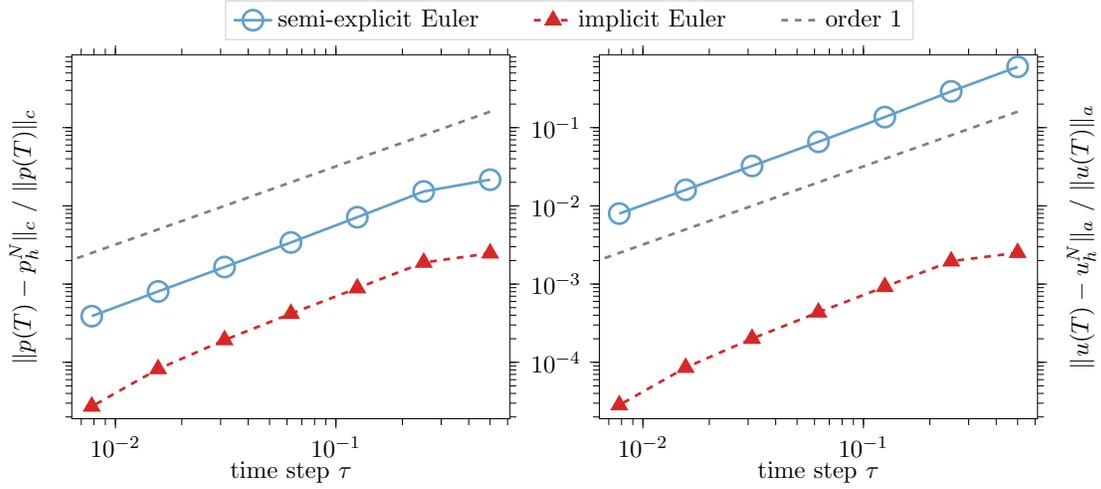

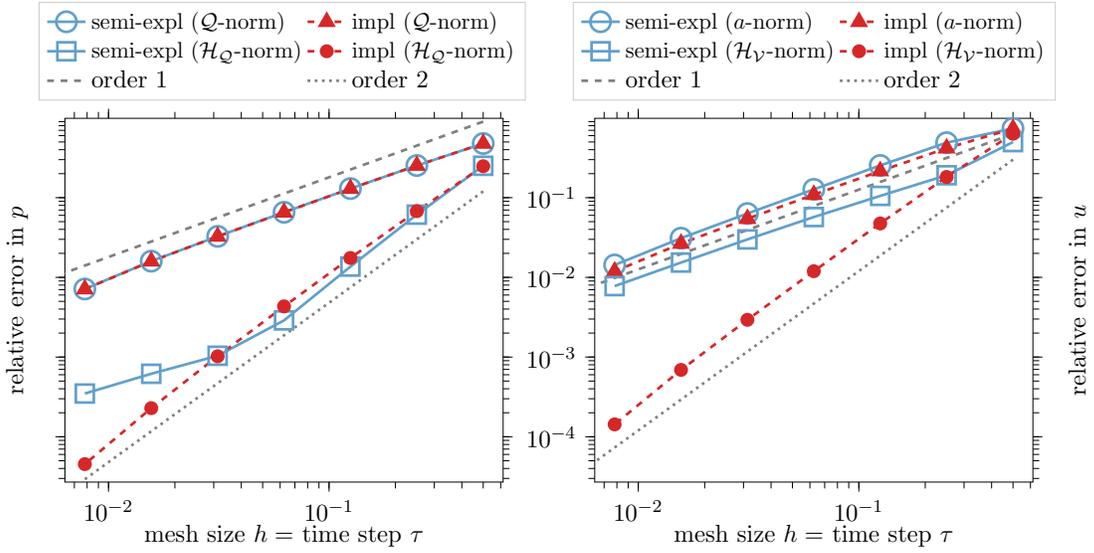
\begin{figure}
	\scalebox{0.85}{
		% This file was created by tikzplotlib v0.8.2.
\begin{tikzpicture}

\begin{axis}[
legend cell align={left},
legend columns = 2,
legend style={at={(-0.06,1.32)}, anchor=north west, draw=white!80.0!black},
log basis x={10},
log basis y={10},
tick align=outside,
tick pos=both,
x grid style={white!69.01960784313725!black},
xlabel={mesh size~$h=$ time step $\tau$},
xmin=0.0063457218465331, 
xmax=0.615572206672458,
xmode=log,
xtick style={color=black},
y grid style={white!69.01960784313725!black},
ylabel={relative error in $p$},
yticklabels = {},
y label style={at={(.07,0.5)}},
ymin=0.000027, ymax=0.99,
ymode=log,
ytick style={color=black}
]
\addplot [very thick, color0!70!white, mark=o, mark size=4.5, mark options={solid}]
table {%
0.5 0.478017183019004
0.25 0.253040343893049
0.125 0.130016064657321
0.0625 0.0655740192793406
0.03125 0.0326934019242939
0.015625 0.0159684606317647
0.0078125 0.00714879550731608
};
\addlegendentry{\small semi-expl ($\Q$-norm)}

\addplot [very thick, color3, dashed, mark=triangle*, mark size=3.5, mark options={solid}]
table {%
0.5 0.477982978231281
0.25 0.252979920221976
0.125 0.129771536601947
0.0625 0.0653922777359294
0.03125 0.0325897482774623
0.015625 0.0159132334316036
0.0078125 0.00711785140159712
};
\addlegendentry{\small impl ($\Q$-norm)}

\addplot [very thick, color0!70!white, mark=square, mark size=4.0, mark options={solid}]
table {%
0.5 0.251081258610354
0.25 0.0612436085222484
0.125 0.0136470266258805
0.0625 0.00288429469527747
0.03125 0.00103825824982263
0.015625 0.00061695381968987
0.0078125 0.000348508063831516
};
\addlegendentry{\small semi-expl ($\cHQ$-norm)}

\addplot [very thick, color3, dashed, mark=*, mark size=2.5, mark options={solid}]
table {%
0.5 0.24801276559204
0.25 0.0677780598465261
0.125 0.0175226176558726
0.0625 0.00432368350056836
0.03125 0.00102313266663078
0.015625 0.000228437953296862
0.0078125 4.53434100640046e-05
};
\addlegendentry{\small impl ($\cHQ$-norm)}

\addplot [very thick, gray, dashed]
table {%
0.5 0.9
0.005 0.009
};
\addlegendentry{order 1}

\addplot [very thick, gray, dotted]
table {%
0.5 0.12
0.005 0.000012
};
\addlegendentry{order 2}

\end{axis}
\end{tikzpicture}
		\hspace{-1.3em}
		% This file was created by tikzplotlib v0.8.2.
\begin{tikzpicture}

\begin{axis}[
legend cell align={left},
legend columns = 2,
legend style={at={(-0.05,1.32)}, anchor=north west, draw=white!80.0!black},
log basis x={10},
log basis y={10},
tick align=outside,
tick pos=both,
x grid style={white!69.01960784313725!black},
xlabel={mesh size~$h=$ time step $\tau$},
xmin=0.0063457218465331, 
xmax=0.615572206672458,
xmode=log,
xtick style={color=black},
y grid style={white!69.01960784313725!black},
ylabel={relative error in $u$},
y label style={at={(1.28,0.5)}},
ymin=0.000027, ymax=0.99,
ymode=log,
ytick style={color=black}
]
\addplot [very thick, color0!70!white, mark=o, mark size=4.5, mark options={solid}]
table {%
0.5 0.739469100932092
0.25 0.489588997041763
0.125 0.253740930878212
0.0625 0.12715657433769
0.03125 0.0632473969242848
0.015625 0.0310238115775371
0.0078125 0.0143092752394005
};
\addlegendentry{\small semi-expl ($a$-norm)}

\addplot [very thick, color3, dashed, mark=triangle*, mark size=3.5, mark options={solid}]
table {%
0.5 0.734124937901243
0.25 0.417212452357853
0.125 0.215857007530659
0.0625 0.108945185949293
0.03125 0.0543588246369921
0.015625 0.0265634233245604
0.0078125 0.0118865310555705
};
\addlegendentry{\small impl ($a$-norm)}

\addplot [very thick, color0!70!white, mark=square, mark size=4.0, mark options={solid}]
table {%
0.5 0.498382837419683
0.25 0.191072804473293
0.125 0.105249319168778
0.0625 0.0571253226972161
0.03125 0.0299355018436014
0.015625 0.0153385577591744
0.0078125 0.0077663650831154
};
\addlegendentry{\small semi-expl ($\cHV$-norm)}

\addplot [very thick, color3, dashed, mark=*, mark size=2.5, mark options={solid}]
table {%
0.5 0.634187848736239
0.25 0.181820249551559
0.125 0.0474605347842002
0.0625 0.0119489442716283
0.03125 0.00293311174991634
0.015625 0.00068996014206436
0.0078125 0.000142909547137742
};
\addlegendentry{\small impl ($\cHV$-norm)}

\addplot [very thick, gray, dashed]
table {%
0.5 0.63
0.005 0.0063
};
\addlegendentry{order 1}

\addplot [very thick, gray, dotted]
table {%
0.5 0.3
0.005 0.00003
};
\addlegendentry{order 2}

\end{axis}
\end{tikzpicture}
	}
	\caption{Relative error in $p$ (left, measured in the $\Q$ and $\cHQ$-norm) and $u$ (right, measured in the $a$ and $\cHV$-norm) in the example of Section~\ref{ss:network} at the final time $T$ for varying $\tau = h$. %The dashed and the dotted lines indicate orders 1 and 2, respectively.
	} 
	\label{fig:ex1_h}
\end{figure}

For completeness, we also present the errors of a simultaneous refinement in space and time in Figure~\ref{fig:ex1_h}. We observe first-order rates for both schemes in $u$ and $p$ when measured in the $H^1$-norms. In particular, the schemes are very close due to the fact that the spatial error is more dominant compared to the temporal one. In such a scenario, the semi-explicit scheme is very beneficial. For the weaker $L^2$-norms, the plots indicate second-order convergence which -- in the case of the semi-explicit scheme -- reduces to a first-order rate once the temporal error takes the dominant part. The results are in line with classical approximation results of first-order finite element approximations in space.
%
%
%=============================================================================
\subsection{Kozeny-Carman model}\label{ss:KozenyCarman}

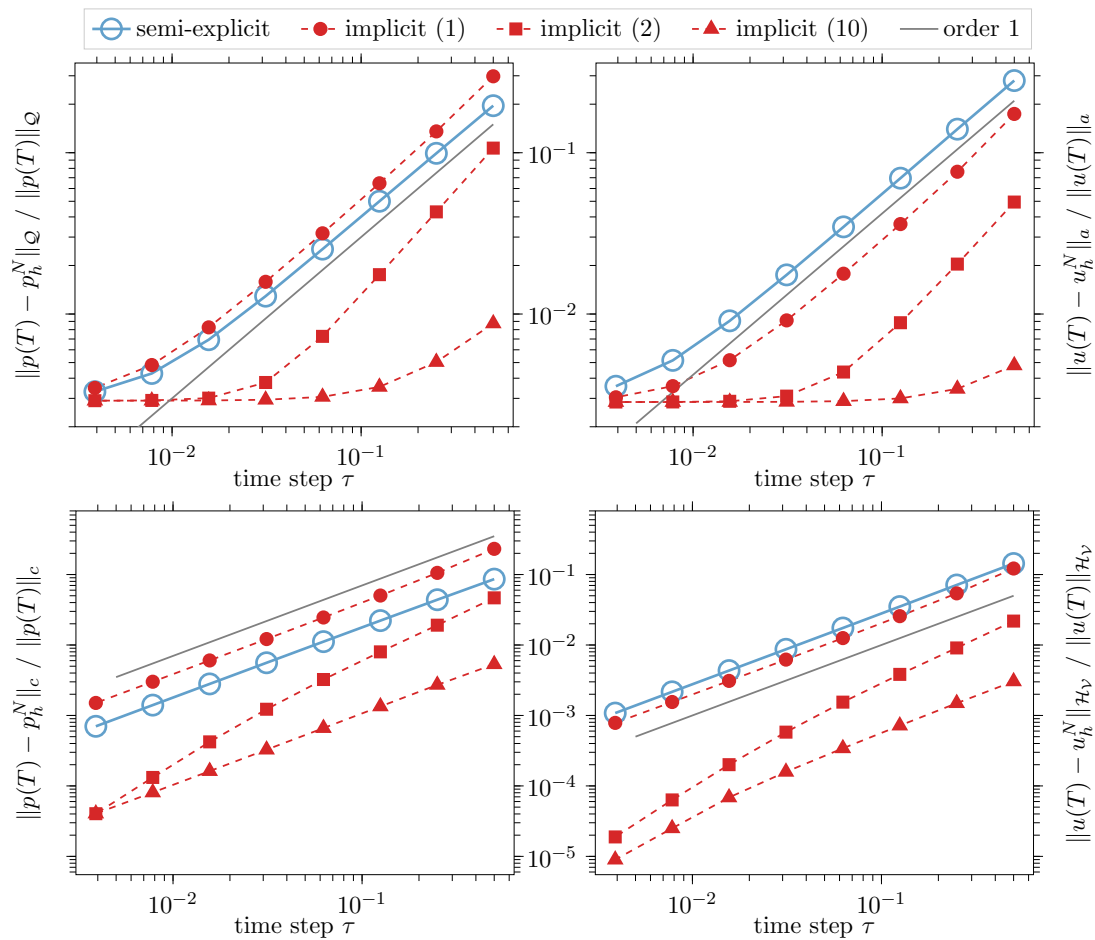
\begin{figure}
	\scalebox{0.85}{
		% This file was created by tikzplotlib v0.8.2.
\begin{tikzpicture}

\begin{axis}[
legend cell align={left},
legend columns = 5,
legend style={at={(0.02,1.15)}, anchor=north west, draw=white!80.0!black},
log basis x={10},
log basis y={10},
tick align=outside,
tick pos=both,
x grid style={white!69.01960784313725!black},
xlabel={time step \(\displaystyle \tau\)},
xmin=0.00306478163240919, xmax=0.637280313659631,
xmode=log,
xtick style={color=black},
y grid style={white!69.01960784313725!black},
ylabel={$\|p(T)-p_h^N\|_{\Q}\  /\ \|p(T)\|_{\Q}$},
yticklabels = {},
y label style={at={(.07,0.5)}},
ymin=0.002, ymax=0.36,
ymode=log,
ytick style={color=black}
]

\addplot [very thick, color0!70!white, mark=o, mark size=4.5, mark options={solid}]
table {%
0.5 0.195758890931729
0.25 0.0991649761588126
0.125 0.0500274728469483
0.0625 0.0252544207298423
0.03125 0.0129040685017434
0.015625 0.0069299128533811
0.0078125 0.00428026916188703
0.00390625 0.0032997849600618
};
\addlegendentry{semi-explicit$\quad$}

\addplot [thick, color3, dashed, mark=*, mark size=2.75, mark options={solid}]
table {%
0.5 0.297814500935466
0.25 0.135612449004966
0.125 0.0646603029987231
0.0625 0.0316569218659776
0.03125 0.0158396120480743
0.015625 0.00826536873853306
0.0078125 0.00482798659383292
0.00390625 0.00348324130556057
};
\addlegendentry{implicit (1)$\quad$}

\addplot [thick, color3, dashed, mark=square*, mark size=2.5, mark options={solid}]
table {%
0.5 0.106834241552282
0.25 0.0429845967113111
0.125 0.0175422812659384
0.0625 0.00727060085103093
0.03125 0.00376365942812409
0.015625 0.00301048786876026
0.0078125 0.00291237208077752
0.00390625 0.00290216551372455
};
\addlegendentry{implicit (2)$\quad$}

\addplot [thick, color3, dashed, mark=triangle*, mark size=3.5, mark options={solid}]
table {%
	0.5 0.00874122703573463
	0.25 0.00503923624101026
	0.125 0.00353346117604677
	0.0625 0.00306451150934983
	0.03125 0.0029412299380578
	0.015625 0.00291070085191509
	0.0078125 0.00290327776410781
	0.00390625 0.00290150653462787
};
\addlegendentry{implicit (10)$\quad$}

\addplot [thick, gray]
table {%
	0.5 0.15
	0.005 0.0015
};
\addlegendentry{order 1}
\end{axis}

\end{tikzpicture}
		\hspace{-8.4cm}
		% This file was created by tikzplotlib v0.8.2.
\begin{tikzpicture}

\begin{axis}[
legend cell align={left},
legend style={at={(0.03,0.97)}, anchor=north west, draw=white!80.0!black},
log basis x={10},
log basis y={10},
tick align=outside,
tick pos=both,
x grid style={white!69.01960784313725!black},
xlabel={time step \(\displaystyle \tau\)},
xmin=0.00306478163240919, xmax=0.637280313659631,
xmode=log,
xtick style={color=black},
y grid style={white!69.01960784313725!black},
ylabel={$\|u(T)-u_h^N\|_{a}\  /\ \|u(T)\|_{a}$},
y label style={at={(1.28,0.5)}},
ymin=0.002, ymax=0.36,
ymode=log,
ytick style={color=black}
]

\addplot [very thick, color0!70!white, mark=o, mark size=4.5, mark options={solid}]
table {%
0.5 0.280164407984489
0.25 0.140015431327361
0.125 0.0695417673656203
0.0625 0.0347065401789834
0.03125 0.0174871323161334
0.015625 0.00907337509985849
0.0078125 0.00515975109193557
0.00390625 0.0035670773429178
};
%\addlegendentry{semi-expl}

\addplot [thick, color3, dashed, mark=*, mark size=2.75, mark options={solid}]
table {%
0.5 0.174027799847367
0.25 0.0762515832240236
0.125 0.0361165300216456
0.0625 0.0177840369308914
0.03125 0.00913170739021825
0.015625 0.00517179864365144
0.0078125 0.00357214963247565
0.00390625 0.00304573306749834
};
%\addlegendentry{impl 1}

\addplot [thick, color3, dashed, mark=square*, mark size=2.5, mark options={solid}]
table {%
0.5 0.0494552363057988
0.25 0.0203968939630028
0.125 0.00883272480589224
0.0625 0.00436835516652356
0.03125 0.00309762329881753
0.015625 0.00287705189256027
0.0078125 0.00284997965420781
0.00390625 0.0028472180605215
};
%\addlegendentry{impl 2}

\addplot [thick, color3, dashed, mark=triangle*, mark size=3.5, mark options={solid}]
table {%
	0.5 0.00480287193700849
	0.25 0.00343644167187772
	0.125 0.00299699363338975
	0.0625 0.00288275030303821
	0.03125 0.00285520294457486
	0.015625 0.00284874070525421
	0.0078125 0.00284728857179397
	0.00390625 0.00284699867398304
};
%\addlegendentry{impl 10}

\addplot [thick, gray]
table {%
	0.5 0.21
	0.005 0.0021
};
%\addlegendentry{order 1}
\end{axis}

\end{tikzpicture}
	}
	\vspace{0.3em}
	\scalebox{0.85}{
		% This file was created by tikzplotlib v0.8.2.
\begin{tikzpicture}

\begin{axis}[
legend cell align={left},
legend style={at={(0.03,0.97)}, anchor=north west, draw=white!80.0!black},
log basis x={10},
log basis y={10},
tick align=outside,
tick pos=both,
x grid style={white!69.01960784313725!black},
xlabel={time step \(\displaystyle \tau\)},
xmin=0.00306478163240919, xmax=0.637280313659631,
xmode=log,
xtick style={color=black},
y grid style={white!69.01960784313725!black},
ylabel={$\|p(T)-p_h^N\|_{c}\  /\ \|p(T)\|_{c}$},
yticklabels = {},
y label style={at={(.07,0.5)}},
ymin=5.5e-06, ymax=0.8,
ymode=log,
ytick style={color=black}
]
\addplot [very thick, color0!70!white, mark=o, mark size=4.5, mark options={solid}]
table {%
0.5 0.0860947789990272
0.25 0.0439178062850828
0.125 0.0221909431653462
0.0625 0.0111597759649673
0.03125 0.00559645273836576
0.015625 0.0028024705828593
0.0078125 0.00140237881528845
0.00390625 0.000701571255877558
};
%\addlegendentry{semi-expl}

\addplot [thick, color3, dashed, mark=*, mark size=2.75, mark options={solid}]
table {%
0.5 0.23159214930759
0.25 0.10560710234203
0.125 0.0503461632214214
0.0625 0.0245779473145363
0.03125 0.0121426320763344
0.015625 0.00603643439139513
0.0078125 0.00301121746149566
0.00390625 0.00150559331155746
};
%\addlegendentry{impl 1}

\addplot [thick, color3, dashed, mark=square*, mark size=2.5, mark options={solid}]
table {%
0.5 0.0466989107754557
0.25 0.0191160822669004
0.125 0.00799319907106565
0.0625 0.0032380199758707
0.03125 0.00122520537052867
0.015625 0.000421940077423715
0.0078125 0.000131875056155365
0.00390625 4.03348517950064e-05
};
%\addlegendentry{impl 2}

\addplot [thick, color3, dashed, mark=triangle*, mark size=3.5, mark options={solid}]
table {%
	0.5 0.00532907487108501
	0.25 0.00271550999502894
	0.125 0.00134264829467187
	0.0625 0.00066133893434869
	0.03125 0.000326939330547271
	0.015625 0.000162200654107671
	0.0078125 8.07107895865761e-05
	0.00390625 4.04925465982497e-05
};
%\addlegendentry{impl 10}

\addplot [thick, gray]
table {%
	0.5 0.35
	0.005 0.0035
};
%\addlegendentry{order 1}
\end{axis}

\end{tikzpicture}
		\hspace{-1.2em}
		% This file was created by tikzplotlib v0.8.2.
\begin{tikzpicture}

\begin{axis}[
legend cell align={left},
legend style={at={(0.03,0.97)}, anchor=north west, draw=white!80.0!black},
log basis x={10},
log basis y={10},
tick align=outside,
tick pos=both,
x grid style={white!69.01960784313725!black},
xlabel={time step \(\displaystyle \tau\)},
xmin=0.00306478163240919, xmax=0.637280313659631,
xmode=log,
xtick style={color=black},
y grid style={white!69.01960784313725!black},
ylabel={$\|u(T)-u_h^N\|_{\cHV}\  /\ \|u(T)\|_{\cHV}$},
y label style={at={(1.28,0.5)}},
ymin=5.5e-06, ymax=0.8,
ymode=log,
ytick style={color=black}
]
\addplot [very thick, color0!70!white, mark=o, mark size=4.5, mark options={solid}]
table {%
0.5 0.143346683685183
0.25 0.071081202726522
0.125 0.0351117555691134
0.0625 0.0174374674319739
0.03125 0.00868790440010258
0.015625 0.00433544798771761
0.0078125 0.0021648586753855
0.00390625 0.00108105213856154
};
%\addlegendentry{semi-expl}

\addplot [thick, color3, dashed, mark=*, mark size=2.75, mark options={solid}]
table {%
0.5 0.121779298330972
0.25 0.0540777296347439
0.125 0.0256594124180353
0.0625 0.0125348511246793
0.03125 0.00620551538005323
0.015625 0.00309489639718068
0.0078125 0.00155262431124385
0.00390625 0.000784721480258117
};
%\addlegendentry{impl 1}

\addplot [thick, color3, dashed, mark=square*, mark size=2.5, mark options={solid}]
table {%
0.5 0.021851791615831
0.25 0.00908470344962249
0.125 0.00382934375445533
0.0625 0.00154579039466402
0.03125 0.000579843710425998
0.015625 0.000200264536188307
0.0078125 6.32834267032144e-05
0.00390625 1.88797221410868e-05
};
%\addlegendentry{impl 2}

\addplot [thick, color3, dashed, mark=triangle*, mark size=3.5, mark options={solid}]
table {%
	0.5 0.00304370437775387
	0.25 0.00149214928867459
	0.125 0.000717731946531882
	0.0625 0.00034151107818188
	0.03125 0.000158355328703976
	0.015625 6.85777385110321e-05
	0.0078125 2.50108881839797e-05
	0.00390625 8.98140582055686e-06
};
%\addlegendentry{impl 10}

\addplot [thick, gray]
table {%
	0.5 0.05
	0.005 0.0005
};
%\addlegendentry{order 1}
\end{axis}

\end{tikzpicture}
	}
	\caption{Relative $H^1$-errors (top) and $L^2$-errors (bottom) at final time $T$ for the semi-explicit method and implicit approximations with different maximal Picard steps (1,2, or 10) for fixed $h = 2^{-8}$ and varying $\tau$. 	}
	\label{fig:ex3_t_L2}
\end{figure}

For our second example, we consider a Kozeny-Carman permeability as explained in Example~\ref{exp:dilatationDependentPermeability}. We set the involved coefficients as 
\begin{equation*}
\rho_0 = 0.5,\quad c_s = -0.75,\quad C_s = 0.75,\quad \alpha = \lambda = \mu = \kappa_0 = \nu = M = T = 1. 
\end{equation*}
The right-hand sides $f$, $g$ and the initial data $p^0$ are chosen such that
\begin{equation*}
p(x,t) = t\,\sin(\pi x_1)\,\sin(\pi x_2),\quad 
u(x,t) = \frac 16\,\big[1\ \ 1\big]^T\exp(-t)\,\sin(\pi x_1)\,\sin(\pi x_2)
\end{equation*}
is the corresponding exact solution. In this case, we have $C_d^2 = \alpha^2 = 1$ and $c_a\,c_c = \mu/M = 1$ such that the coupling condition of \Cref{ass:weakCoupling} is only just fulfilled.  

This setup is taken from the example considered in~\cite{CaoCM13}, where only an implicit scheme in combination with a Picard-type iteration is used. Here, we compare our semi-explicit approach with three different implicit approaches, where either maximal one, two, or ten Picard iterations are performed for each time step. This means, the iteration either stops after the maximal amount of steps or when the relative residual tolerance of $10^{-9}$ is reached. 
Using the norms $\|\cdot\|_a$ and $\|\cdot\|_{\Q}$ for the variables $u$ and $p$, respectively, we present in Figure~\ref{fig:ex3_t_L2}~(top) the relative errors for these four methods and different values of $\tau$ for a fixed mesh size $h = 2^{-8}$. We observe a linear convergence behavior in $\tau$ for both $u$ and $p$, but the curves stagnate when the spatial error starts to dominate. The level of stagnation is similar for all curves due to the fact that the spatial errors behave similarly for these approaches as already observed in the first example (cf.~Figure~\ref{fig:ex1_h}). Note that the semi-explicit scheme is closest to the implicit approximation with just one Picard step. 

Apart from the behavior with respect to these stronger norms, we also present the errors in the $\cHV$-norm and $\cHQ$-norm, see Figure~\ref{fig:ex3_t_L2}~(bottom). Due to the smaller spatial error in these weaker norms, the error curves do not show a stagnation and display a linear convergence behavior in $\tau$. 

Finally, another important aspect with the semi-explicit approach is the speed-up due to the decoupling and the automatic linearization. To this end, we compare the run times of the semi-explicit scheme with the ones of the implicit Picard-type approaches with maximal one, two, or ten inner iteration steps. 
We emphasize that the semi-explicit scheme is significantly faster than any of the implicit schemes if the same time step $\tau$ is considered. However, as it can be observed in Figure~\ref{fig:ex3_t_L2}, the implicit schemes with two or more Picard steps lead to smaller errors (e.g., up to roughly a factor $20$ in $p$ and $100$ in u, respectively, for the implicit approach with 10 Picard steps). Therefore, we do not compare these methods using the same time steps but rather choose different $\tau$ such that the relative errors measured in the norm~$\vvvert(v,q)\vvvert^2 := \|v\|^2_{a} + \|q\|^2_{c}$ are comparable. 
	
We present the corresponding run times in Table~\ref{tab:runtimeT}. In this comparison, the semi-explicit scheme achieves the smallest error, still with faster run times compared to the other methods. Note that the computation time of the semi-explicit method can be heavily exploited for weaker coupling strengths. This can be observed in Table~\ref{tab:runtimeT2}, which shows the run times and errors for the same problem when $\mu = 10$ and $M = 0.1$ are chosen instead. In this setup, the coupling strength of the equations is weaker and therefore the speed-up more apparent. Note that this effect can also be observed in the following example, where the differences between the implicit and the semi-explicit scheme become smaller when the coupling gets weaker. 

\newcolumntype{P}{C{1.2cm}}
\newcolumntype{Q}{C{1.6cm}}
\begin{table}[t]
	\caption{Run time comparison (in seconds) for fixed $h=2^{-8}$ and the coefficients given in Section~\ref{ss:KozenyCarman}.}\label{tab:runtimeT}
	\begin{tabular}{lPQQ}
		%\toprule
		 & $\tau$ & error & run time 
		\\\toprule[.05cm]%\midrule
		semi-explicit & $2^{-6}$ & $0.00697$ & $663.86$   
%		\\[-0.2em]
%		 & $2^{-8}$ & $0.00330$ & $2689.98$ 
		\\\midrule[.05cm]
		implicit (max.~10 Picard steps) & $2^{-1}$ & $0.00870$ & $761.63$ 
%		\\[-0.2em]
%		& $2^{-3}$ & $0.00353$ & $2938.65$
		\\\midrule
		implicit (max.~2 Picard steps) & $2^{-4}$ & $0.00724$ & $1334.73$
%		\\[-0.2em] 
%		& $2^{-5}$ & $0.00375$ & $2623.76$ 
		\\\midrule
		implicit (1 Picard step) & $2^{-6}$ & $0.00823$ & $2836.90$
%		\\[-0.2em]
%		& $2^{-8}$ & $0.00348$ & $2938.65$ 
		%\\\bottomrule
	\end{tabular}
\end{table}
\begin{table}[t]
	\caption{Run time comparison (in seconds) for fixed $h=2^{-8}$ and with adjusted coefficients~$\mu = 10$ and $M = 0.1$.}\label{tab:runtimeT2}
	\begin{tabular}{lPQQ}
		%\toprule
		& $\tau$ & error & run time %& \textbf{speed-up factor}
		\\\toprule[.05cm]%\midrule
		semi-explicit & $2^{-4}$ & $0.00452$ & $170.16$  
%		\\[-0.2em]
%		& $2^{-6}$ & $0.00302$ & $832.79$ 
		\\\midrule[.05cm]
		implicit (max.~10 Picard steps) & $2^{-1}$ & $0.00621$ & $495.52$ 
%		\\[-0.2em]
%		& $2^{-3}$ & $0.00332$ & $1680.97$ 
		\\\midrule
		implicit (max.~2 Picard steps) & $2^{-2}$ & $0.00524$ & $335.32$
%		\\[-0.2em] 
%		& $2^{-4}$ & $0.00304$ & $1365.97$ 
		\\\midrule
		implicit (1 Picard step) & $2^{-6}$ & $0.00790$ & $3008.32$ 
%		\\[-0.2em]
%		& $2^{-8}$ & $0.00343$ & $12123.23$ 
		%\\\bottomrule
	\end{tabular}
\end{table}
%
%
%=============================================================================
\subsection{Sharpness of the coupling condition}\label{ss:sharp}
The third and final example is devoted to showing that the weak coupling condition in Assumption~\ref{ass:weakCoupling} is indeed necessary and rather sharp. For this, we consider a simple poroelasticity test with coefficients
\begin{equation*}
	\lambda = \mu = \kappa_0 = \nu = M = T = 1. 
\end{equation*}
Moreover, we consider right-hand sides $f \equiv 0$, $g(x,t) = 5\,\cos(0.5\pi t) + \sin(0.5\pi t)$ and varying $\alpha$ to assess the stability of the semi-explicit discretization. Since $C_d^2 = \alpha^2$ and $c_a\,c_c = \mu/M = 1$, the coupling condition reads $\alpha \leq 1$. 
The quadratic nonlinear permeability in this example is given by 
\begin{equation*}
	\kappa(s) \coloneqq \left\{
	\begin{aligned}
	&\,\kappa_0\,c_s^2, \quad && \rho(s) \leq c_s,\\
	&\,\kappa_0\, (\rho(s))^2, \quad && c_s < \rho(s) < C_s,\\
	&\,\kappa_0\, C_s^2\quad && \rho(s) \geq C_s,
	\end{aligned}
	\right.
\end{equation*}
where $\rho(s) = \rho_0 + (1-\rho_0)s$, $\rho_0 = 0.4$, $c_s = 0.01$, and $C_s  = 0.75$. 

In Figure~\ref{fig:ex23}, we present the errors between the implicit and the semi-explicit discretization for~$h=2^{-4}$, different time step sizes $\tau \in \{2^{-2},2^{-3},2^{-4},2^{-5}\}$, and multiple coefficients~$\alpha \in [0.05,5]$. Note that the implicit scheme is unconditionally stable. 
We observe that for $\tau = 2^{-5}$, the semi-explicit scheme becomes unstable for $\alpha \approx 2$ and slightly later for the other time step sizes. This is generally in line with Assumption~\ref{ass:weakCoupling}, but it seems that, in practice, the condition may be slightly relaxed for larger $\tau$. 
\begin{figure}
	\scalebox{0.8}{
		% This file was created by tikzplotlib v0.9.6.
\begin{tikzpicture}

\begin{axis}[
width=5.0in,
height=3.0in,
legend cell align={left},
legend style={fill opacity=0.8, draw opacity=1, text opacity=1, at={(0.03,0.97)}, anchor=north west, draw=white!80!black},
log basis y={10},
tick align=outside,
tick pos=left,
x grid style={white!69.0196078431373!black},
xlabel={coupling coefficient $\alpha$},
xmin=-0.1975, xmax=5.2475,
xtick style={color=black},
y grid style={white!69.0196078431373!black},
ylabel={relative error},
y label style={at={(-.05,0.5)}},
ymin=0.0007, ymax=1000,
ymode=log,
ytick style={color=black}
]
\addplot [thick, red!80!black, dashed]
table {%
0.05 0.00424797156524939
0.1 0.00856152704128158
0.15 0.0129279935452944
0.2 0.0173346107983261
0.25 0.0217683609306665
0.3 0.0262157773438259
0.35 0.0306627381310963
0.4 0.0350942451919297
0.45 0.039494194886086
0.5 0.0438451499916614
0.55 0.0481281108350209
0.6 0.0523223075839025
0.65 0.0564050176097729
0.7 0.0603514257162957
0.75 0.0641345451336544
0.8 0.0677252240805711
0.85 0.0710922643212757
0.9 0.0742027049197727
0.95 0.0770223212887505
1 0.0795164353813363
1.05 0.0816511751937057
1.1 0.0833953864219358
1.15 0.0847235191711308
1.2 0.0856199647046557
1.25 0.0860855615091197
1.3 0.086147273077442
1.35 0.0858723186760925
1.4 0.0853880535888815
1.45 0.084907923614115
1.5 0.0847604972231046
1.55 0.0854110148224641
1.6 0.0874536103394992
1.65 0.0915469125206431
1.7 0.0982862513164621
1.75 0.108054991853197
1.8 0.120926392508902
1.85 0.136648364340351
1.9 0.154681086650475
1.95 0.17424360295175
2 0.19435377311307
2.05 0.21387155159001
2.1 0.231422272022827
2.15 0.2458327567349
2.2 0.257366311980512
2.25 0.264958887198115
2.3 0.270663435499776
2.35 0.275838951118672
2.4 0.281342781698751
2.45 0.285876000077321
2.5 0.290334536129354
2.55 0.294648863754165
2.6 0.299110520249699
2.65 0.303323842602447
2.7 0.307455281218402
2.75 0.311639212691827
2.8 0.315685532160223
2.85 0.319369436185869
2.9 0.322275147309103
2.95 0.323800055399111
3 0.322982610931806
3.05 0.321088390079993
3.1 0.318242758721031
3.15 0.314628402757523
3.2 0.310863510668897
3.25 0.306191475666274
3.3 0.299769771902704
3.35 0.292736032981056
3.4 0.285523315145575
3.45 0.278707370126173
3.5 0.27404993774347
3.55 0.277768877632481
3.6 0.303308268437033
3.65 0.377679453036626
3.7 0.545440649855966
3.75 0.847077749021127
3.8 1.15726400207486
3.85 1.32827480088826
3.9 1.42701949779336
3.95 1.47940253302614
4 1.53773717604509
4.05 1.74047445062115
4.1 1.91195116388434
4.15 2.33665357686434
4.2 2.80351013759025
4.25 2.99978202548749
4.3 3.23627814589042
4.35 3.46398561236363
4.4 4.19599449700434
4.45 4.92509281801437
4.5 5.58987129686923
4.55 5.86389539692647
4.6 6.73517754875996
4.65 7.79566605181011
4.7 8.5849890258345
4.75 10.156830334088
4.8 10.9762014535014
4.85 11.8716231627507
4.9 13.187860631849
4.95 14.4855460182103
5 15.876039645381
};
\addlegendentry{$\tau = 2^{-2}$}
\addplot [thick, green!50!black, dashed]
table {%
0.05 0.00284152329944412
0.1 0.00571275953445161
0.15 0.00861006963118017
0.2 0.0115293427908157
0.25 0.0144657174660678
0.3 0.017413291092075
0.35 0.0203648023470779
0.4 0.0233112915486331
0.45 0.0262417449984531
0.5 0.0291427285263951
0.55 0.0319980284822643
0.6 0.0347883105484767
0.65 0.0374908360356057
0.7 0.0400792650569314
0.75 0.0425236007251569
0.8 0.044790337939994
0.85 0.0468429121844355
0.9 0.0486425707683899
0.95 0.0501498732629683
1 0.0513271121092921
1.05 0.0521421520461156
1.1 0.0525744932554355
1.15 0.0526248549064882
1.2 0.0523302451640396
1.25 0.0517870542023994
1.3 0.0511841194596031
1.35 0.0508428264520367
1.4 0.0512466996513557
1.45 0.0530182157358613
1.5 0.0567946446851737
1.55 0.06302326247958
1.6 0.0718006742025583
1.65 0.0828538021799756
1.7 0.0956179944864936
1.75 0.109321083723879
1.8 0.123042199304666
1.85 0.135760326783077
1.9 0.146412675170629
1.95 0.153618301077474
2 0.157274763755626
2.05 0.159109532165235
2.1 0.161717051006155
2.15 0.164195359108706
2.2 0.166067335214991
2.25 0.167206961242866
2.3 0.167275959739676
2.35 0.166160497581927
2.4 0.163352570465121
2.45 0.158818367878359
2.5 0.152955059199407
2.55 0.147005018966972
2.6 0.142909426195081
2.65 0.143248896949644
2.7 0.148972480746646
2.75 0.163564001725396
2.8 0.209593482189491
2.85 0.355126861294615
2.9 0.716610034230612
2.95 0.855568802753357
3 1.28867329744222
3.05 1.98174004909804
3.1 3.34029804062177
3.15 4.6659671818402
3.2 7.32855249363575
3.25 10.5277914547649
3.3 13.8069460182212
3.35 16.939610453691
3.4 22.9165130627237
3.45 33.0228380840169
3.5 48.5011234001282
3.55 66.8470145228879
3.6 76.409296063065
3.65 98.7635020606304
3.7 133.574334847256
3.75 155.991153843827
3.8 191.768205024662
3.85 239.565564543834
3.9 304.915758896151
3.95 344.933960118188
4 392.456962091184
4.05 474.678512602428
4.1 552.941905062502
4.15 669.799087719819
4.2 812.9753882123
4.25 983.257057111418
4.3 1160.65385392035
4.35 1498.35456161245
4.4 1840.80863283582
4.45 2229.43629988041
4.5 2703.58110367121
4.55 3403.3767218991
4.6 4361.3710094259
4.65 5301.1026476722
4.7 6368.89314405735
4.75 7335.88825622567
4.8 8745.47794067015
4.85 10479.7739985495
4.9 12528.7078250976
4.95 14757.1265987442
5 17226.4613001155
};
\addlegendentry{$\tau = 2^{-3}$}
\addplot [thick, yellow!50!orange, dashed]
table {%
0.05 0.00162287334992688
0.1 0.00326133892984124
0.15 0.00491519911184054
0.2 0.00658368679444486
0.25 0.00826530234399977
0.3 0.0099576167454421
0.35 0.0116570561884713
0.4 0.013358641928181
0.45 0.0150557084155023
0.5 0.0167395855291996
0.55 0.0183992750634748
0.6 0.0200211300770007
0.65 0.0215885801455244
0.7 0.0230819344514471
0.75 0.0244783463612238
0.8 0.0257520157070187
0.85 0.0268747553775311
0.9 0.0278171006163189
0.95 0.0285502070296619
1 0.02904894780671
1.05 0.0292968640079916
1.1 0.0292940720999271
1.15 0.0290698278348141
1.2 0.028701972154615
1.25 0.0283444841502029
1.3 0.0282580496663298
1.35 0.0288212015266076
1.4 0.0304781418139214
1.45 0.0336003990574087
1.5 0.0383342609721127
1.55 0.0445484881464044
1.6 0.0518862272190412
1.65 0.0598389197884734
1.7 0.0677981060253338
1.75 0.0750909731492909
1.8 0.0810145901904475
1.85 0.0846959187859382
1.9 0.0855172579700863
1.95 0.0865601548197623
2 0.0873309126158238
2.05 0.0865587559191812
2.1 0.0821430937303295
2.15 0.0738973513537902
2.2 0.064231686921896
2.25 0.0829928350747624
2.3 0.182662296232188
2.35 0.399653106756266
2.4 0.912649935891672
2.45 2.74526325862423
2.5 7.74484912031276
2.55 17.9088705926604
2.6 40.7273968675336
2.65 78.8459508964203
2.7 150.420585214244
2.75 303.246234800156
2.8 542.766178050729
2.85 965.684837514984
2.9 1613.95514452279
2.95 2994.02300590123
3 5322.44852743343
3.05 9116.97150168554
3.1 15646.01911445
3.15 25338.3338942875
3.2 40723.9343454537
3.25 63604.8154567006
3.3 102734.29864641
3.35 164876.810146924
3.4 246883.129299862
3.45 396428.138840778
3.5 680133.263764437
3.55 1075139.81805827
3.6 1672889.1502474
3.65 2652374.24490928
3.7 4042333.67730017
3.75 6089233.24622494
3.8 9169832.30425987
3.85 13627756.5018545
3.9 19866340.5498504
3.95 28616147.8218445
4 41091066.0981717
4.05 58887807.4762312
4.1 85061054.3223914
4.15 110664199.368896
4.2 158417575.998209
4.25 227200632.092846
4.3 324113661.7483
4.35 460890596.09673
4.4 646622034.78367
4.45 903884195.617543
4.5 1271982020.92955
4.55 1730961483.97191
4.6 2414764656.67529
4.65 3277257197.99993
4.7 4546328623.63004
4.75 6435766463.4332
4.8 8906596401.03708
4.85 12296568595.5925
4.9 16911222479.404
4.95 23243570730.5717
5 31871560017.0175
};
\addlegendentry{$\tau = 2^{-4}$}
\addplot [thick, blue!70!black, dashed]
table {%
0.05 0.00086458980499611
0.1 0.00173754135178292
0.15 0.00261929729976109
0.2 0.00350989832272686
0.25 0.00440890364232175
0.3 0.00531528366705542
0.35 0.00622729317882887
0.4 0.00714230766485116
0.45 0.00805664974561904
0.5 0.0089653769311486
0.55 0.00986205806589999
0.6 0.010738554623026
0.65 0.011584814329509
0.7 0.0123887326419308
0.75 0.0131361275266625
0.8 0.0138109031645053
0.85 0.0143955062667384
0.9 0.0148718288834747
0.95 0.015222751177514
1 0.0154346799656963
1.05 0.0155016086180834
1.1 0.0154315808263076
1.15 0.0152568577896424
1.2 0.0150490110838821
1.25 0.0149377517323346
1.3 0.0151236355293276
1.35 0.0158596426304011
1.4 0.0173781006883296
1.45 0.0197920014385309
1.5 0.0230472735437366
1.55 0.0269474818152735
1.6 0.0312018936871214
1.65 0.0354620289025095
1.7 0.039347048396108
1.75 0.0424651323681412
1.8 0.0443408303080481
1.85 0.0444729701914462
1.9 0.0444846646226268
1.95 0.0409489844307186
2 0.0328098662350668
2.05 0.150485405740548
2.1 0.823303001851184
2.15 8.3180913346015
2.2 54.1361524628036
2.25 276.078162910412
2.3 1241.77289237822
2.35 5459.43271620183
2.4 21598.3604374785
2.45 74376.0732623602
2.5 290295.944939951
2.55 911797.20135388
2.6 3396060.02063393
2.65 10938063.9278668
2.7 30848440.1292234
2.75 96481439.5518829
2.8 330014812.565942
2.85 975495390.114833
2.9 2466320806.44068
2.95 6950951654.53055
3 16729156680.9291
3.05 46783446453.2398
3.1 148429111970.432
3.15 397392813253.416
3.2 1205882481660.46
3.25 2632294814938.96
3.3 6636191197975.73
3.35 16389889052069.9
3.4 32833625086783.6
3.45 79826041365031.2
3.5 192117323640648
3.55 401970117613898
3.6 949239402360552
3.65 2.21346422595942e+15
3.7 5.12163380788955e+15
3.75 1.33597103327869e+16
3.8 2.9944838992034e+16
3.85 6.64516753652689e+16
3.9 1.4612143762963e+17
3.95 3.18524692143405e+17
4 6.85073457726305e+17
4.05 1.8470156620053e+18
4.1 3.90403520764249e+18
4.15 8.08219432266842e+18
4.2 1.3337944836139e+19
4.25 2.69481128611955e+19
4.3 4.69656831017611e+19
4.35 9.5069888846928e+19
4.4 1.90802466604861e+20
4.45 3.80375049130403e+20
4.5 7.5425003834555e+20
4.55 1.47355631818065e+21
4.6 2.31597828921756e+21
4.65 4.48249120072762e+21
4.7 8.62922618482334e+21
4.75 1.65220551471311e+22
4.8 3.145267240629e+22
4.85 5.94457895141532e+22
4.9 1.11593960512978e+23
4.95 2.07675495273982e+23
5 3.58802963874268e+23
};
\addlegendentry{$\tau = 2^{-5}$}
\end{axis}

\end{tikzpicture}
	}
	\caption{Illustration of the practical coupling condition for the step sizes $\tau = 2^{-2}, \dots, 2^{-5}$. The plots contain the relative error in the $\vvvert\cdot\vvvert$-norm at the end of the time interval.	}
	\label{fig:ex23}
\end{figure}
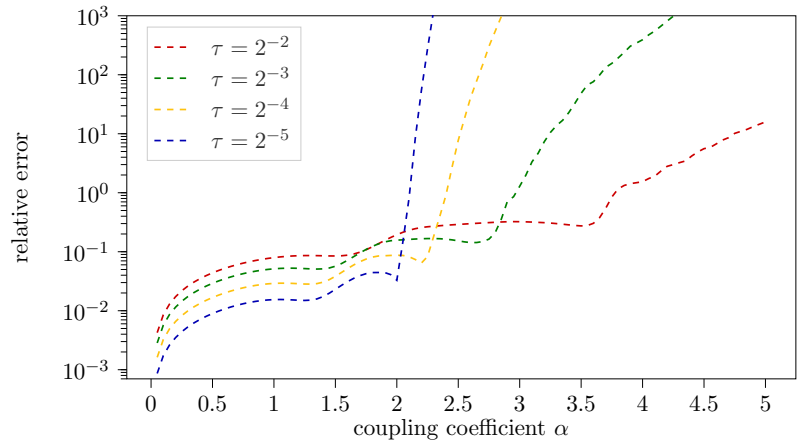
%
%
%=============================================================================
\subsection{Discussion of results}
In view of the results of the previous subsections, we can first record that the semi-explicit scheme shows the same order of convergence as an implicit one but -- if the coupling condition is barely fulfilled or slightly violated -- leads to larger errors for the same time step sizes. This, however, is outweighed by the speed-up that the semi-explicit scheme provides with its decoupling of the poroelastic equations. If the coupling between the equations is relatively weak, the semi-explicit scheme significantly outperforms the implicit ones, in general. Let us also emphasize that the semi-explicit discretization is very beneficial in a setting where $\tau$ and $h$ are refined simultaneously, since the spatial error seems to be rather dominant in such cases (cf.~Figure~\ref{fig:ex1_h}).

Finally, we mention that a coupling condition as in~\Cref{ass:weakCoupling} is indeed necessary when using the semi-explicit scheme as investigated in Section~\ref{ss:sharp}, but the condition appears to be less critical in practice. 
%
%
%=============================================================================
%=========  Conclusions
%=============================================================================
\section{Conclusions}
Within this paper, we have proposed the use of a semi-explicit discretization scheme for the problem of nonlinear poroelasticity with displacement-dependent permeability. In the setting of a weak coupling of the two involved equations, the scheme allows a decoupling and, in particular, results directly in a one-step linearization approach to treat the nonlinearity. We haven proven first-order convergence in time and illustrated the performance of the approach in multiple numerical examples. The semi-explicit method provides a speed-up compared to a classical implicit scheme with an inner iteration, especially for a~relatively weak coupling of the equations.
%
%
%=============================================================================
%=========  Acknowledgments
%=============================================================================
\section*{Acknowledgments} 
R.~Altmann acknowledges the support of the Deutsche Forschungsgemeinschaft (DFG, German Research Foundation) through the project 467107679. 
R.~Maier gratefully acknowledges support by the G\"oran Gustafsson Foundation for Research in Natural Sciences and Medicine. %
%
%
%=============================================================================
%=========  Bibs
%=============================================================================
%\bibliographystyle{alpha}
%\bibliography{references}
\newcommand{\etalchar}[1]{$^{#1}$}

\end{document}